\documentclass[a4paper,12pt]{article}
\usepackage{euscript,amssymb,amsfonts,amsmath,amsthm,mathtools,graphicx}
\usepackage{mathtext}
\usepackage{color}
\usepackage{mathrsfs}
\usepackage[left=3cm,right=1.5cm,top=2cm,bottom=2cm,bindingoffset=0cm]{geometry}
\usepackage{enumerate}
\usepackage{multirow}

\usepackage[cp1251]{inputenc}
\usepackage[T2A]{fontenc}
\usepackage[english]{babel}

\newtheorem{Theorem}{\indent \sc Theorem}
\newtheorem{Lemma}{\indent \sc Lemma}
\newtheorem{cor}{\indent \sc Corollary}

\theoremstyle{remark}
\newtheorem*{Remark}{\indent \sc Remark}

\newcommand{\eps}{\varepsilon}
\newcommand{\R}{\mathbb R}
\newcommand{\N}{\mathbb N}
\newcommand{\E}{{\sf E}}
\newcommand{\D}{{\sf Var}}
\newcommand{\Prob}{{\sf P}}

\renewcommand{\le}{\leqslant}
\renewcommand{\ge}{\geqslant}
\newcommand{\I}{{\bf1}}
\newcommand{\abs}[1]{\left|#1\right|}

\newcommand{\wave}[1]{\widetilde{#1}}
\newcommand{\G}{\mathcal G}
\newcommand{\abe}{C_{\textrm{\tiny AB}}} % asymptotically best constant
\newcommand{\lowaex}{{\underline{C_{\textrm{\tiny AE}}}}}% lower asymptotically exact constant
\newcommand{\aex}{C_{\textrm{\tiny AE}}} % asymptotically exact constant
\newcommand{\upaex}{\overline{C}_{\textrm{\tiny AE}}} % upper asymptotically exact constant
\newcommand{\upaexcond}{C_{\textrm{\tiny AE}}^*} % condinitional upper aex
\newcommand{\qt}{\sigma^2}
\newcommand{\gopt}{\gamma_*}
\newcommand{\scE}{{\textrm{\tiny E}}}
\newcommand{\scR}{{\textrm{\tiny R}}}
\newcommand{\es}{L_{\scE,n}}
\newcommand{\roz}{L_{\scR,n}}
\newcommand{\essC}{A_\scE}
\newcommand{\rozC}{A_\scR}
\newcommand{\CE}{C_\scE}
\newcommand{\CR}{C_\scR}
\newcommand{\aexess}{\essC^\textrm{\tiny AE}}
\newcommand{\aexroz}{\rozC^\textrm{\tiny AE}}
\newcommand{\osiM}{\Lambda_n}
\newcommand{\essM}{M_n}
\newcommand{\gmmalow}{\Upsilon}
\newcommand{\gmmahigh}{\Gamma}%%%%%%%%%%%%%%%%%%%%%%%%%%%%%%%%%%%%%%%%%%%%%%%%%%%%

\date{}

\title{Asymptotically exact constants in natural convergence rate estimates in the Lindeberg theorem\thanks{This research was supported by the Russian Foundation for Basic Research (project 19-07-01220-a) and by the Ministry for Education and Science of Russia (grant No.~MD--5748.2021.1.1).}}

\author{Ruslan Gabdullin\thanks{Lomonosov Moscow State University, Moscow, Russia}, Vladimir Makarenko\thanks{Lomonosov Moscow State University, Moscow, Russia}, and Irina Shevtsova\footnote{Corresponding author. Hangzhou Dianzi University, Hangzhou, China. Lomonosov Moscow State University, Moscow, Russia. Institute of Informatics Problems of Federal Research Center ``Computer Science and Control,'' Russian Academy of Sciences, Moscow, Russia. E-mail: ishevtsova@cs.msu.ru}}

\begin{document}

\maketitle

\begin{abstract}
Following (Shevtsova, 2010) we introduce detailed classification of the asymptotically exact constants in natural estimates of the rate of convergence in the Lindeberg central limit theorem, namely in  Esseen's, Rozovskii's, and Wang--Ahmad's inequalities and their structural improvements obtained in our previous works. The above inequalities involve algebraic truncated third-order moments and the classical Lindeberg fraction and assume finiteness only the second-order moments of random summands. We present lower bounds for the introduced asymptotically exact constants as well as for the universal and for the most optimistic constants which turn to be not far from the upper ones. 
\end{abstract}

{\bf keywords:} central limit theorem, Lindeberg's theorem, normal approximation, asymptotically exact constant, asymptotically best constant, uniform distance, Lindeberg fraction, truncated moment, absolute constant

\section{Introduction}

Let $X_1, X_2, \ldots, X_n$ be independent random variables  (r.v.') with distribution functions (d.f.'s) $F_k(x) = \Prob(X_k < x)$, $x\in\R,$ expectations $\E X_k = 0$, variances $\sigma_k^2 = \D X_k$, $k=1,\ldots,n$, and such that
$$
B_n^2:= \sum_{k=1}^n \sigma_k^2>0.
$$
For $n=1,2,\ldots$ denote
$$
S_n = X_1 + X_2 + \cdots + X_n, \quad %B_n^2 = \D S_n = \sum_{k=1}^n \sigma_k^2, \quad 
\wave{S}_n = \frac{S_n-\E S_n}{\sqrt{\D S_n}} = \sum_{k=1}^n \frac{X_k}{B_n},
$$
$$
%\overline{F}_n(x) = \Prob(\wave{S}_n < x), \quad 
\Phi(x) = \frac{1}{\sqrt{2\pi}} \int_{-\infty}^x e^{-t^2/2}\,dt, \ x \in \R,
\quad
\Delta_n= \Delta_n(F_1,\ldots,F_n)= \sup_{x \in \R} \abs{\Prob(\wave{S}_n < x)- \Phi(x)},
$$
$$
\sigma_k^2(z) = \E X_k^2 \I(\abs{X_k} \ge z), \quad \mu_k(z) = \E X_k^3 \I(\abs{X_k} < z),\quad k=1,\ldots,n
$$
$$
\essM(z):=\frac1{B_n^3}\sum_{k=1}^n \mu_k(zB_n) = \frac1{B_n^3}\sum_{k=1}^n \E X_k^3\I(|X_k|<zB_n),
$$
$$
\osiM(z)\coloneqq  \frac1{B_n^3}\sum_{k=1}^n\E|X_k|^3\I(|X_k|<z B_n),
$$
$$
L_n(z) = \frac{1}{B_n^2} \sum_{k=1}^n \sigma_k^2(zB_n) = \frac{1}{B_n^2} \sum_{k=1}^n \E X_k^2\I (\abs{X_k} \ge zB_n), \quad z>0.
$$
The function $L_n(\,\cdot\,)$ is called the \textit{Lindeberg fraction}. It is easy to see that $|\essM(z)|\le\osiM(z)$, $z>0$. In case of independent identically distributed (i.i.d.) r.v.'s $X_1,\ldots,X_n$ we denote their common d.f. by $F$ and write $\Delta_n(F):=\Delta_n(F,\ldots,F)$. %and set $\Delta_n(F):=\Delta_n(F,\ldots,F)$.

%\subsection{История}

In~\cite{GabdullinMakarenkoShevtsova2018} it was proved that
\begin{equation} \label{EsseenTypeIneq2018}
\Delta_n \le \essC(\eps,\gamma)\sup_{0 < z < \eps} \left\{ \gamma\abs{\essM(z)} + zL_n(z) \right\},
\end{equation}
\begin{equation} \label{RozovskiiTypeIneq2018}
\Delta_n \le \rozC(\eps,\gamma)\Big(\gamma\abs{\essM(\eps)} + \sup_{0 < z < \eps} zL_n(z) \Big),\quad \eps,\,\gamma>0,\quad n\in\N,
\end{equation}
where the functions $\essC(\eps,\gamma),\,\rozC(\eps,\gamma)$ depend only on $\eps$ and $\gamma$ (that is, they turn into absolute constants as soon as $\eps$ and $\gamma$ are fixed), both are monotonically non-increasing with respect to~$\gamma>0$,  and  $\essC(\eps,\gamma)$ is also non-increasing with respect to~$\eps>0$. The question on the boundedness of $\rozC(\eps,\gamma)$ as $\eps\to\infty$ is still open, while $\essC(0+,\gamma)=\rozC(0+,\gamma)=\infty$  for every $\gamma>0$ and $\essC(\eps,0+)=\rozC(\eps,0+)=\infty$ for every $\eps>0$. To avoid ambiguity, in what follows by constants appearing here in various inequalities we mean their exact values; in particular, in majorizing expressions --- their least possible values.
%assuring the validity of the corresponding inequality for all values of paraemter considered
Upper bounds for the constants $\essC(\eps,\gamma)$ and $\rozC(\eps,\gamma)$ for some $\eps$ and $\gamma$ computed in~\cite{GabdullinMakarenkoShevtsova2018}  are presented in tables~\ref{Tab:EsseenC1(L)} and~\ref{Tab:RozovskiiC1(L)}, respectively. Here the symbol $\gopt$ stands for the point of minimum of the upper bound $\rozC(\eps,\gamma)$, obtained within the framework of the method used in~\cite{GabdullinMakarenkoShevtsova2018}, that is, the bound $\rozC(\eps,\gamma)$  found in~\cite{GabdullinMakarenkoShevtsova2018} remains constant as $\gamma\ge\gopt$ grows for every fixed $\eps>0$. More precisely, the quantity~$\gamma_*$ is defined as follows:
$$
\gopt=1/\sqrt{6\varkappa}=0.5599\ldots,
$$
where
$$
\varkappa=x^{-2}\sqrt{(\cos x-1+x^2/2)^2+(\sin x-x)^2}\Big|_{x=x_0}=0.5315\ldots,
$$
$x_0=5.487414\ldots$ is the unique root of the equation
$$
8(\cos x - 1) + 8x\sin x - 4x^2\cos x - x^3\sin x = 0,\quad x\in(\pi,2\pi).
$$

\begin{table}[h]
\begin{center}
\begin{tabular}{||c|c|c||c|c|c||c|c|c||}
\hline
$\eps$&$\gamma$&$\essC(\eps,\gamma)$& $\eps$&$\gamma$&$\essC(\eps,\gamma)$& $\eps$&$\gamma$&$\essC(\eps,\gamma)$
\\ \hline %E
$1.21$&$0.2$&$2.8904$&$\infty$&$\gopt$&$2.6919$&$2.65$&$4$&$2.6500$\\ \hline %E
$1.24$&$0.2$&$2.8900$&$1$&$0.72$&$2.7298$&$2.74$&$3$&$2.6500$\\ \hline %E
$\infty$&$0.2$&$2.8846$&$1$&$\infty$&$2.7286$&$3.13$&$2$&$2.6500$\\ \hline %E
$1.76$&$0.4$&$2.7360$&$4.35$&$1$&$2.6600$&$4$&$1.62$&$2.6500$\\ \hline %E
$5.94$&$0.4$&$2.7300$&$\infty$&$1$&$2.6588$&$5.37$&$1.5$&$2.6500$\\ \hline %E
$\infty$&$0.4$&$2.7299$&$\infty$&$0.97$&$2.6599$&$\infty$&$1.43$&$2.6500$\\ \hline %E
$1$&$\gopt$&$2.7367$&$2.56$&$\infty$&$2.6500$&$\infty$&$\infty$&$2.6409$\\ \hline %E
$1.87$&$\gopt$&$2.6999$&$2.62$&$5$&$2.6500$&$0+$&$\forall$&$\infty$\\ \hline %E
\end{tabular}
\end{center}
\caption{Upper bounds for the constants $\essC(\eps,\gamma)$ from inequality~\eqref{EsseenTypeIneq2018}, $\gopt=0.5599\ldots$}
\label{Tab:EsseenC1(L)}
\end{table}

\begin{table}[h]
\begin{center}
\begin{tabular}{||c|c|c||c|c|c||}
\hline
$\eps$&$\gamma$&$\rozC(\eps,\gamma)$&
$\eps$&$\gamma$&$\rozC(\eps,\gamma)$
\\ \hline
$1.21$&$0.2$&$2.8700$&$1.99$&$\gopt$&$2.6600$\\ \hline %R
$5.39$&$0.2$&$2.8635$&$2.12$&$\gopt$&$2.6593$\\ \hline %R
$1.76$&$0.4$&$2.6999$&$3$&$\gopt$&$2.6769$\\ \hline %R
$2.63$&$0.4$&$2.6933$&$5$&$\gopt$&$2.7562$\\ \hline %R
$0.5$&$\gopt$&$3.0396$&$0+$&$\forall$&$\infty$\\ \hline %R
$1$&$\gopt$&$2.7286$&&&\\ \hline %R
\end{tabular}
\end{center}
\caption{Upper bounds for the constants $\rozC(\eps,\gamma)$  from inequality~\eqref{RozovskiiTypeIneq2018}, $\gopt=0.5599\ldots$}
\label{Tab:RozovskiiC1(L)}
\end{table}

%\begin{table}[h]
%\begin{center}
%\begin{tabular}{||c|c|c||}
%\hline
%$\eps$&$\gamma$&$C_R(\eps,\gamma)\le$
%\\ \hline
%$1.21$&$0.2$&$2.86991$\\ \hline %R
%$5.39$&$0.2$&$2.86343$\\ \hline %R
%$1.76$&$0.4$&$2.69985$\\ \hline %R
%$2.63$&$0.4$&$2.69323$\\ \hline %R
%$0.5$&$\gopt$&$3.03953$\\ \hline %R
%$1$&$\gopt$&$2.72857$\\ \hline %R
%$1.99$&$\gopt$&$2.65991$\\ \hline %R
%$2.12$&$\gopt$&$2.65925$\\ \hline %R
%$3$&$\gopt$&$2.67687$\\ \hline %R
%$5$&$\gopt$&$2.75611$\\ \hline %R
%\end{tabular}
%\end{center}
%\caption{Upper bounds for $C_R(\eps,\gamma)$ in~\eqref{RozovskiiTypeIneq2018}. Recall that $\gopt=0.5599\ldots\ .$}
%\label{Tab:RozovskiiC1(L)}
%\end{table}

In the same paper~\cite{GabdullinMakarenkoShevtsova2018} there were found sharpened upper bounds for the constants $\essC(\eps,\gamma)$ and $\rozC(\eps,\gamma)$ provided that the corresponding fractions
$$
\es(\eps, \gamma):=\sup_{0 < z < \eps} \left\{ \gamma\abs{\essM(z)} + zL_n(z) \right\},\quad \roz(\eps, \gamma):=\gamma\abs{\essM(\eps)} + \sup_{0 < z < \eps} zL_n(z).
$$
take small values. In particular, there were introduced asymptotically exact constants 
\begin{equation}\label{aexCessDef}
\aexess(\eps,\gamma):= \limsup_{\ell \to 0} \sup_{n,F_1,\ldots,F_n} \left\{ \frac{\Delta_n(F_1,\ldots,F_n)}{\ell}\colon \es(\eps,\gamma)=\ell \right\},\quad \eps,\ \gamma>0,
\end{equation}
\begin{equation}\label{aexCrozDef}
\aexroz(\eps,\gamma):=\limsup_{\ell \to 0} \sup_{n,F_1,\ldots,F_n} \left\{ \frac{\Delta_n(F_1,\ldots,F_n)}{\ell}\colon \roz(\eps,\gamma)=\ell \right\},\quad \eps,\ \gamma>0,
\end{equation}
and the following upper bounds  were obtained for them:
\begin{eqnarray}\label{aexess(eps,gamma)<=}
\aexess(\eps,\gamma)&\le&\frac{4}{\sqrt{2\pi}}+ \frac{1}{\pi}  \bigg[\frac{\varkappa}{\eps}\gmmalow\Big(1,\frac{t_\gamma^2}{2\eps^2}\Big) +\frac{\eps}{12}\gmmalow\Big(2,\frac{t_\gamma^2}{2\eps^2}\Big) +  
\frac{\sqrt{2(6\varkappa\gamma^2+1)}} {6\gamma}\, \gmmahigh\Big(\frac{3}{2},\frac{t_\gamma^2}{2\eps^2}\Big) \bigg],
%\eqqcolon \widehat\aexess(\eps,\gamma),
\\\nonumber
\aexroz(\eps,\gamma)&\le&\frac{4}{\sqrt{2\pi}}+ \frac{1}{\pi}  \Big[
\frac{\varkappa}{\eps}
\gmmalow\Big(1,\frac{t_{1,\gamma}^{2}}{2\eps^2}\Big) +\frac{\eps}{12}
\gmmalow\Big(2,\frac{t_{1,\gamma}^{2}}{2\eps^2}\Big)+
\frac{\eps}{6}\gmmahigh\Big(2,\frac{t_{2,\gamma}^{2}}{2\eps^2}\Big)+
\\&&
\label{aexroz(eps,gamma)<=}
+\frac{\sqrt2}{6\gamma}\Big(\frac{\sqrt\pi}{2}
-\gmmalow\Big(\frac32,\frac{t_{1,\gamma}^{2}}{2\eps^2}\Big)
-\gmmahigh\Big(\frac32,\frac{t_{2,\gamma}^{2}}{2\eps^2}\Big)\Big)
\Big],
%\eqqcolon \widehat\aexroz(\eps,\gamma),
\end{eqnarray}
where $\gmmahigh(r,x) \coloneqq \int_{x}^\infty t^{r-1} e^{-t} dt,$
$\gmmalow(r,x)  \coloneqq\int_0^{x} t^{r-1} e^{-t} dt =\gmmahigh(r)
- \gmmahigh(r,x),$ $r,x>0,$ are the upper and the lower gamma-functions, respectively,
$$
t_\gamma\coloneqq \tfrac2\gamma\big(\sqrt{(\gamma/\gopt)^2+1}-1\big),\quad
t_{2,\gamma}=2\max\big\{\gamma^{-1},\gopt^{-1}\big\}, \quad
t_{1,\gamma}\coloneqq t_{2,\gamma}\big(1-\sqrt{(1-(\gamma/\gopt)^2)_+}\,\big),
$$
$\varkappa=0.5315\ldots$ were $\gopt=0.5599\ldots$ defined above. The values of the upper bounds of the asymptotically exact constants $\aexess(\eps,\gamma)$ and $\aexroz(\eps,\gamma)$ in~\eqref{aexess(eps,gamma)<=} and~\eqref{aexroz(eps,gamma)<=} for some $\eps>0$ and $\gamma>0$ are given in the third and the sixth columns of table~\ref{Tab:Ess+RozAEX(eps,gamma)<=}, respectively. 

\begin{table}[h!]
\begin{center}
\begin{tabular}{||c|c|c||c|c|c||}
\hline
$\eps$&$\gamma$&$\vphantom{\displaystyle\frac12}\aexess(\eps,\gamma)\le$&
$\eps$&$\gamma$&$\aexroz(\eps,\gamma)\le$\\ \hline
$0.6$&$0.3$&$1.92245$&$1.21$&$0.2$&$1.93474$\\ \hline %R
$1.21$&$0.2$&$1.95457$&$1.89$&$0.2$&$1.92998$\\ \hline %R
$2.06$&$0.2$&$1.94999$&$2.77$&$0.2$&$1.92890$\\ \hline %R
$\infty$&$0.2$&$1.94879$&$5.39$&$0.2$&$1.95832$\\ \hline %R
$1.48$&$0.4$&$1.80997$&$1.41$&$0.4$&$1.77974$\\ \hline %R
$\infty$&$0.4$&$1.80005$&$1.76$&$0.4$&$1.77249$\\ \hline %R
$1.89$&$\gopt$&$1.77136$&$1.99$&$0.4$&$1.77128$\\ \hline %R
$2.03$&$\gopt$&$1.76995$&$2.63$&$0.4$&$1.77841$\\ \hline %R
$\infty$&$\gopt$&$1.76370$&$0.5$&$\gopt$&$1.94743$\\ \hline %R
$1$&$\gopt$&$1.80596$&$1$&$\gopt$&$1.79154$\\ \hline %R
$1$&$0.67$&$1.79961$&$1.52$&$\gopt$&$1.74995$\\ \hline %R
$1$&$\infty$&$1.79149$&$1.89$&$\gopt$&$1.74383$\\ \hline %R
$2.24$&$1$&$1.73996$&$1.99$&$\gopt$&$1.74412$\\ \hline %R
$\infty$&$1$&$1.73186$&$2.12$&$\gopt$&$1.74542$\\ \hline %R
$3.07$&$\infty$&$1.71998$&$3$&$\gopt$&$1.77092$\\ \hline %R
$3.2$&$5$&$1.71997$&$5$&$\gopt$&$1.86500$\\ \hline %R
$3.28$&$4$&$1.71999$&&&\\\hline
$4$&$2.4$&$1.71998$&&&\\\hline
$5$&$2.06$&$1.71997$&&&\\\hline
$5.37$&$2$&$1.72000$&&&\\\hline
$\infty$&$1.83$&$1.71995$&&&\\\hline
$\infty$&$\infty$&$1.71451$&&&\\\hline
\end{tabular}
\end{center}
\caption{Values of the upper bounds of the asymptotically exact constants $\aexess(\eps,\gamma)$ and $\aexroz(\eps,\gamma)$ in~\eqref{aexess(eps,gamma)<=} and~\eqref{aexroz(eps,gamma)<=} for some $\eps>0$ and $\gamma>0$. Recall that $\gopt=0.5599\ldots\ .$}
\label{Tab:Ess+RozAEX(eps,gamma)<=}
\end{table}

Moreover, in~\cite{GabdullinMakarenkoShevtsova2018}  it was shown that the asymptotically exact constants $A_{\bullet}^\textrm{\tiny AE}\in\{\aexess,\aexroz\}$ are unbounded as $\gamma\to0$:
$$
A_{\bullet}^\textrm{\tiny AE}(\eps,\gamma)\ge 
%\sup_{F_1=\ldots=F_n}\limsup_{n\to\infty} \frac{\Delta_n(F_1,\ldots,F_n)}
\sup_{F}\limsup_{n\to\infty} \frac{\Delta_n(F)}
{L_{\bullet,n}(\eps,\gamma)}\to\infty, \quad\gamma\to0,\quad \forall\eps>0.
$$

Let us note that estimate~\eqref{RozovskiiTypeIneq2018} with $\eps=\gamma=1$  coincides with the Rozovskii inequality~\cite[Corollary\,1]{Rozovskii1974} and establishes an upper bound for the appearing absolute constant $\rozC(1,1)\le\rozC(1,\gopt)\le2.73$. Estimate~\eqref{EsseenTypeIneq2018} with $\eps=\gamma=1$ and  $\eps\to\infty$, $\gamma=1$ improves both Esseen's inequalities from~\cite{Esseen1969}, where the absolute value sign and the least upper bound with respect to $z\in(0,\eps)$  stand inside the sum in comparison with~\eqref{EsseenTypeIneq2018}. In particular, estimate~\eqref{EsseenTypeIneq2018} yields upper bounds for the absolute constants in Esseen's inequalities  $\essC(1,1)\le\essC(1,0.72)\le2.73$ and $\essC(\infty,1)\le\essC(\infty,0.97)\le2.66$ which were remaining unknown for a long time. Esseen's inequality where the least upper bound is taken over a bounded range (see~\eqref{EsseenTypeIneq2018} with $\eps=\gamma=1$) yields, in its turn, the Osipov inequality~\cite{Osipov1966}:
\begin{multline}\label{OsipovIneq}
\Delta_n \le \essC(1,1)\sup_{0 < z < 1} \left\{\abs{\essM(z)} + zL_n(z) \right\}\le \essC(1,1)\sup_{0<z<1} \left\{\osiM(z)+zL_n(z)\right\}=
\\
=\essC(1,1)(\osiM(1)+L_n(1))= \essC(1,1)\inf_{\eps>0}(\osiM(\eps)+L_n(\eps))
\end{multline}
(for details see~\cite{GabdullinMakarenkoShevtsova2018}). The latest bound obtained by Osipov~\cite{Osipov1966} with \textit{some} constant $\essC(1,1)$ (whose best known value $1.87$ is published in~\cite{KorolevDorofeyeva2017}) yields, in its turn, the Lindeberg theorem: indeed, under the Lindeberg condition $\sup\limits_{\eps>0}\lim\limits_{n\to\infty}L_n(\eps)=0$ and with the account of  $\osiM(\eps)\le\eps$, from~\eqref{OsipovIneq} we have
$$
\Delta_n\le\essC(1,1)\inf_{\eps>0}(\eps+L_n(\eps))\to0,\quad n\to\infty.
$$
Hence, by Feller's theorem, in case of uniformly infinitesimal random summands (in particular, in the i.i.d. case) the right- and the left-hand sides of~\eqref{OsipovIneq} are either both infinitesimal or both do not tend to zero. According to the terminology, introduced by Zolotarev~\cite{Zolotarev1986}, such convergence rate estimates are called \textit{natural}. Together with~\eqref{EsseenTypeIneq2018}, inequality~\eqref{RozovskiiTypeIneq2018} is surely a natural convergence rate estimate in the Lindeberg--Feller theorem under the additional assumption of existence of such an $\eps_0>0$ that $\essM(\eps_0)=0$ for all sufficiently large $n$ (see, e.g.,~\cite{GabdullinMakarenkoShevtsova2020}), in particular, if the r.v.'s $X_1,\ldots,X_n$ have symmetric distributions.

Let us also note that Esseen-type inequality~\eqref{EsseenTypeIneq2018} not only links the \textit{criteria of convergence} with the \textit{rate of convergence}, as Osipov's inequality does, but also provides a \textit{numerical demonstration} of the Ibragimov's criteria~\cite{Ibragimov1966} of the rate of convergence in the CLT to be of order order $\mathcal O(n^{-1/2})$. According to~\cite{Ibragimov1966}, in the i.i.d. case we have $\Delta_n=\mathcal O(n^{-1/2})$ as $n\to\infty$ if and only if $\max\{|\mu_1(z)|,z\qt_1(z)\}=\mathcal O(1)$ as $z\to\infty.$ Inequality~\eqref{EsseenTypeIneq2018} trivially yields the sufficiency of the Ibragimov condition to $\Delta_n=\mathcal O(n^{-1/2})$ as $n\to\infty$.

Let us denote by $\G$ a set of all non-decreasing functions  $g\colon[0,\infty)\to[0,\infty)$ such that ${g(z)>0}$ for $z>0$ and $z / g(z)$ is also non-decreasing for $z>0$.  The set~$\G$ was initially introduced by Katz~\cite{Katz1963} and used later in the works~\cite{Petrov1965,KorolevPopov2012,KorolevDorofeyeva2017,GabdullinMakarenkoShevtsova2019,GabdullinMakarenkoShevtsova2019-2,GabdullinMakarenkoShevtsova2020}. In~\cite{GabdullinMakarenkoShevtsova2019} it was proved that
\begin{itemize}
\item[\rm(i)] For every function $g \in \G$ and $a > 0$
\begin{equation}\label{gmin_g_gmax_ineq}
g_0(z, a):=\min\left\{\frac{z}{a}, 1\right\}\ \le\ \frac{g(z)}{g(a)}\ \le\ \max\left\{\frac{z}{a}, 1\right\}:= g_1(z, a),\quad z>0,
\end{equation}
with $g_0(\,\cdot\,, a),$ $g_1(\,\cdot\,, a) \in \G.$
\item[\rm(ii)] Every function from $\G$ is continuous on $(0,\infty)$.
\end{itemize}
Property~\eqref{gmin_g_gmax_ineq} means that every function from $\G$ is asymptotically (as its argument goes to infinity) between a constant and a linear function. For example, besides $g_0$, $g_1,$ the class $\G$ also includes the following functions:
$$
g_{\textsc{c}}(z)\equiv 1,\quad g_*(z)=z,\quad c \cdot z^{\delta}, \quad c\cdot g(z),\qquad z>0,
$$
for all $c > 0$, $\delta \in [0,1]$ and $g\in\G$.

For $g\in\G$ we set 
\begin{multline} \label{L^3_def}
\es(g,\eps, \gamma) = \frac{1}{B_n^2 g(B_n)}  \sup_{0 < z < \eps B_n} \frac{g(z)}{z} \bigg\{ \gamma \bigg|\sum_{k = 1}^n \mu_k(z)\bigg| + z\sum_{k = 1}^n\sigma_k^2(z) \bigg\} 
\\
= \sup_{0 < z < \eps} \frac{g(zB_n)}{zg(B_n)} \left( \gamma\abs{\essM(z)} + zL_n(z) \right),
\end{multline}
\begin{multline} \label{wave_L^3_def}
\roz(g,\eps, \gamma) = \frac{1}{B_n^2 g(B_n)}\bigg(\gamma\,\frac{g(\eps B_n)}{\eps B_n}\bigg|\sum_{k=1}^n \mu_k(\eps B_n)\bigg|+\sup_{0<z < \eps B_n} g(z)\sum_{k=1}^n \sigma_k^2(z)\bigg) 
\\
= \gamma\,\frac{g(\eps B_n)}{\eps g(B_n)} \abs{\essM(\eps)} + \sup_{0 < z < \eps}\frac{g(zB_n)}{g(B_n)}\,L_n(z).
\end{multline}
Note that the introduced fractions with $g=g_*$ coincide with the fractions in the Esseen- and Rozovskii-type inequalities~\eqref{EsseenTypeIneq2018}, \eqref{RozovskiiTypeIneq2018} considered above:
$$
\es(g_*,\eps,\gamma) = \es(\eps,\gamma),\quad  \roz(g_*,\eps,\gamma) = \roz(\eps,\gamma), \quad \eps,\ \gamma>0.
$$
Inequalities~\eqref{EsseenTypeIneq2018}  and~\eqref{RozovskiiTypeIneq2018}  were generalized in~\cite{GabdullinMakarenkoShevtsova2020} in the following way:
\begin{eqnarray} \label{AW_ineq}
\Delta_n &\le& \CE(\eps, \gamma)\cdot\es(g,\eps, \gamma),
\\\label{Rozovsky_ineq}
\Delta_n &\le& \CR(\eps, \gamma)\cdot\roz(g,\eps, \gamma),\quad \eps,\,\gamma > 0,\quad g \in \G,
\end{eqnarray}
where
$$
%\CR(\eps, \gamma) \le \max\{ \eps, 1 \} \cdot \rozC(\eps, \gamma), 
\begin{array}{rcll}
\CE(\eps, \,\cdot\,)&=&\essC(\eps, \,\cdot\,),\quad\ \eps\in(0,1],&\text{ in particular, } \CE(+0,\,\cdot\,)= \infty,
\\
\CE(\eps, \,\cdot\,)&\le&\essC(1, \,\cdot\,),\quad\ \eps>1;
\\
\CR(\eps,\,\cdot\,)&=&\rozC(\eps,\,\cdot\,),\quad\  \eps\in(0,1],&\text{ in particular, } \CR(+0,\,\cdot\,)= \infty,
\\
\CR(\eps,\,\cdot\,)&\le&\eps\rozC(\eps,\,\cdot\,),\quad \eps>1,
&\CR(\infty,\,\cdot\,) = \infty
\end{array}
$$
(recall that, according to the above convention, all the equalities between the constants including  $\CE(0,\,\cdot\,)=\CR(0,\,\cdot\,)=\CR(\infty,\,\cdot\,) = \infty$ are exact with formal definitions of $\CE(\eps,\gamma)$ and $\CR(\eps,\gamma)$ being given in~\eqref{CE+CR(eps,gamma)defs} below).

It is easy to see that the both fractions $L_{\bullet\,,\,n}\in\{\es,\roz\}$ are invariant with respect to scale transformations of $g\in\G$:
$$
L_{\bullet\,,\,n}(cg,\,\cdot\,,\,\cdot\,) = L_{\bullet\,,\,n}(g,\,\cdot\,,\,\cdot\,),
%\quad \roz(cg,\,\cdot\,,\,\cdot\,)=\roz(c,\,\cdot\,,\,\cdot\,),
\quad c>0.
$$
%\textit{так что при фиксированных $n$ и $F_1,\ldots,F_n$  можно считать, что $g(B_n)=1$.} 
%Кроме того, экстремальные свойства функций 
Moreover, in~\cite[Theorem\,2]{GabdullinMakarenkoShevtsova2020} it was proved that for all $\eps,\gamma>0$
\begin{equation}\label{g_1_Esseen}
1 \le \es(g_1,\eps, \gamma) \le \max\{\eps, 1\}\cdot \max\{\gamma, 1\} ,
\end{equation}
\begin{equation}\label{g_1_Rozovsky}
1 \le \roz(g_1,\eps, \gamma) \le \max\{\eps, 1\} \cdot (\gamma + 1).
\end{equation}
Extreme properties of the functions 
$$
g_0(z):= B_ng_0(z,B_n)=\min\{z, B_n\}, \quad g_1(z):= B_ng_1(z,B_n)=\max\{z, B_n\},\quad z>0,
$$
in~\eqref{gmin_g_gmax_ineq} with $a:=B_n$ yield
\begin{equation}\label{gmin_gmax_ineq}
L_{\bullet\,,\,n}(g_0,\,\cdot\,,\,\cdot\,) \le L_{\bullet\,,\,n}(g,\,\cdot\,,\,\cdot\,) \le L_{\bullet\,,\,n}(g_1,\,\cdot\,,\,\cdot\,),\quad g\in\G,
\end{equation}
for every fixed set of distributions of $X_1,\ldots,X_n$,
%\begin{eqnarray}\label{gmin_gmax_ineq}
%\es(g_0,\,\cdot\,,\,\cdot\,) \le \es(g,\,\cdot\,,\,\cdot\,) \le \es(g_1,\,\cdot\,,\,\cdot\,),\\
%\roz(g_0,\,\cdot\,,\,\cdot\,) \le \roz(g,\,\cdot\,,\,\cdot\,) \le \roz(g_1,\,\cdot\,,\,\cdot\,),
%\end{eqnarray}
so that the extreme values of the constants $\CE$ and $\CR$ in~\eqref{AW_ineq} and~\eqref{Rozovsky_ineq} with fixed~$n$ and $F_1,\ldots,F_n$  are attained at $g=g_0$. Moreover, with the extreme functions~$g$ the fraction $L_{\bullet\,,\,n}\in\{\es,\roz\}$ satisfy the following relations for $\eps\le1$:
\begin{equation}\label{identical_g}
L_{\bullet\,,\,n}(g_0,\eps,\,\cdot\,) = L_{\bullet\,,\,n}(g_*,\eps,\,\cdot\,) = L_{\bullet\,,\,n}(\eps,\,\cdot\,), 
\end{equation}
\begin{equation}\label{const_g}
L_{\bullet\,,\,n}(g_1,\eps,\,\cdot\,) = L_{\bullet\,,\,n}(g_{\textsc{c}},\eps,\,\cdot\,).
\end{equation}

Inequality~\eqref{AW_ineq} also generalizes and improves up to the values of the appearing absolute constant the classical Katz--Petrov inequality~\cite{Katz1963,Petrov1965} (which is equivalent to the Osipov inequality~\cite{Osipov1966}) because of involving the algebraic truncated third order moments instead of the absolutes ones and also a recent result of  Wang and Ahmad~\cite{WangAhmad2016} at the expense of moving the modulus and the least upper bound signs outside of the sum sign. In particular, inequality~\eqref{AW_ineq} established an upper bound of the constant in the Wang--Ahmad inequality $\CE(\infty, 1)\le\essC(1,1)\le2.73$.

A detailed survey and analysis of the relationships between inequalities~\eqref{EsseenTypeIneq2018}, \eqref{RozovskiiTypeIneq2018}, \eqref{AW_ineq}, \eqref{Rozovsky_ineq} with inequalities of Katz~\cite{Katz1963}, Petrov~\cite{Petrov1965}, Osipov~\cite{Osipov1966}, Esseen~\cite{Esseen1969}, Rozovskii~\cite{Rozovskii1974}, and Wang--Ahmad~\cite{WangAhmad2016} can be found in papers~\cite{GabdullinMakarenkoShevtsova2018,GabdullinMakarenkoShevtsova2020}.

The main goal of the present work is construction of the lower bounds of the absolute constants $\CE(\eps,\gamma)$, $\CR(\eps,\gamma)$ in inequalities~\eqref{AW_ineq}, \eqref{Rozovsky_ineq}, and also of the constants $\essC(\eps,\gamma)$, $\rozC(\eps,\gamma)$ in inequalities~\eqref{EsseenTypeIneq2018}, \eqref{RozovskiiTypeIneq2018}, in particular, we show that even in the i.i.d. case
$$
\CE(1,1)=\essC(1,1)> 0.5685,\quad \CR(1,1)=\rozC(1,1)> 0.5685,\quad 
$$
$$
\CE(\infty,1)>0.5685,\quad \essC(\infty,1)>0.3703.
$$
We consider various statements of the problem of construction of the lower bounds, namely, we  introduce a detailed classification of the asymptotically exact constants and construct their lower bounds. As a corollary, we obtain two-sided bounds for the asymptotically exact constants $\aexess(\eps,\gamma)$ and $\aexroz(\eps,\gamma)$ defined in~\eqref{aexCessDef} and~\eqref{aexCrozDef}, in particular, we show that
$$
0.4097\ldots=\frac{\sqrt{10}+3}{6\sqrt{2\pi}}\le\aexess(1,1)\le1.80.
$$
$$
0.3989\ldots=\frac{1}{\sqrt{2\pi}}\le\aexroz(1,1)\le1.80.
$$

The paper is organized  as follows. In Section\,2 we introduce exact, asymptotically exact and asymptotically best constants defining the corresponding statements for the construction of the lower bounds. Sections 4, 5, and 6 are devoted namely to the construction of the lower bounds for the introduced constants. Section 3 contains some auxiliary results which might represent an independent interest, in particular, the values of the fractions $\es(g,\eps,\gamma)$ and $\es(g,\eps,\gamma)$ are found for all $n$, $\eps,\gamma>0$ and some $g\in\G$ in the case where $X_1,\ldots,X_n$ have identical two-point distribution.

\section{Exact, asymptotically exact and asymptotically best constants}

Following~\cite{Kolmogorov1953,Esseen1956,Zolotarev1986,Chistyakov1996,Chistyakov2001a,Chistyakov2001b,Chistyakov2002,Shevtsova2010DAN,Shevtsova2010TVP,Shevtsova2014DAN}, let us define exact, asymptotically exact and asymptotically best constants in inequalities~\eqref{AW_ineq}, \eqref{Rozovsky_ineq}.  Let $\mathcal{F}$ be a set of all d.f.'s with zero means and finite second order moments. Denote
\begin{equation}\label{CE+CR(g,eps,gamma)defs}
\CE(g, \eps, \gamma) = 
\sup_{\substack{F_1,\ldots,F_n\in\mathcal{F}, \\n\in\N\colon B_n>0}} \frac{\Delta_n(F_1,\ldots,F_n)}{\es(g,\eps,\gamma)},\quad  \CR(g,\eps,\gamma) = \sup_{\substack{F_1,\ldots,F_n\in\mathcal{F}, \\n\in\N\colon B_n>0}} \frac{\Delta_n(F_1,\ldots,F_n)}{\roz(g,\eps,\gamma)}.
\end{equation}
Note that the fractions $\es(g,\eps,\gamma)$, $\roz(g,\eps,\gamma)$ also depend on d.f.'s $F_1,\ldots,F_n$, but we omit these arguments for the sake of brevity. 

The constants $\CE(g, \eps, \gamma)$ and $\CR(g, \eps, \gamma)$ are the minimal possible (\textit{exact}) values of the constants $\CE(\eps,\gamma)$ and $\CR(\eps,\gamma)$ in inequalities~\eqref{AW_ineq}, \eqref{Rozovsky_ineq} for the fixed function $g\in\G$, while their universal values $\sup_{g \in \G} \CE(g, \eps, \gamma)$, $\sup_{g \in \G} \CR(g, \eps, \gamma),$ that provide the validity of the inequalities under consideration for \textit{all} $g\in\G$ are called \textit{exact} constants and namely they are the minimal possible (exact) values of the constants $\CE(\eps,\gamma)$ and $\CR(\eps,\gamma)$ in~\eqref{AW_ineq} and~\eqref{Rozovsky_ineq}, respectively. In order not to introduce excess notation and following the above convention, we use namely these exact values for the definitions of $\CE(\eps,\gamma)$ and $\CR(\eps,\gamma)$ in the present work:
\begin{equation}\label{CE+CR(eps,gamma)defs}
\CE(\eps, \gamma) = \sup_{g \in \G} \CE(g, \eps, \gamma), \quad \CR(\eps, \gamma) = \sup_{g \in \G} \CR(g, \eps, \gamma),
\end{equation}
and call them the \textit{exact} constants. Note that
$$
\essC(\eps,\gamma)=\CE(g_*,\eps,\gamma)\le\CE(\eps,\gamma),\quad \rozC(\eps,\gamma)=\CR(g_*,\eps,\gamma)\le\CR(\eps,\gamma),\quad \eps,\gamma>0,
$$
hence, every lower bound for the constants $\essC(\eps,\gamma)$, $\rozC(\eps,\gamma)$ serves as a lower bound for the constants $\CE(\eps,\gamma)$, $\CR(\eps,\gamma)$ as well.

The least upper bounds in~\eqref{CE+CR(g,eps,gamma)defs} are taken without any restrictions on the values of the fractions $\es(g,\eps, \gamma)$, $\roz(g,\eps, \gamma)$, while inequalities~\eqref{AW_ineq}, \eqref{Rozovsky_ineq} represent the most interest with \textit{small} values of these fractions, when the normal approximation is adequate and only the concrete estimates of its accuracy are needed. That is why it is interesting to study not only the \textit{absolute}, but also the \textit{asymptotic} constants  in~\eqref{AW_ineq}, \eqref{Rozovsky_ineq}, some of which were already introduced in~\eqref{aexCessDef}, \eqref{aexCrozDef}. Since we are interested in the lower bounds, we assume that $X_1,\ldots, X_n$ have  identical distributions. Generalization of the definitions introduced below to the non-i.id. case is not difficult  (see, for example, definitions~\eqref{aexCessDef} and~\eqref{aexCrozDef} of the asymptotically exact constants for the general case). 

For each of the fractions $L_n\in\{\es,\roz\}$ appearing in inequalities~\eqref{AW_ineq}, \eqref{Rozovsky_ineq} we define
the \textit{asymptotically best constants}
\begin{equation}\label{ABEdef}
\abe(g, \eps, \gamma) = \sup_{F\in\mathcal{F}} \limsup_{n \rightarrow \infty} \Delta_n(F)/L_n(g,\eps,\gamma),
\end{equation}
the \textit{upper asymptotically exact constants}
\begin{equation} \label{UpAEXdef}
\upaex(g, \eps, \gamma) = \limsup_{n \rightarrow \infty} \sup_{F\in\mathcal{F}} \Delta_n(F)/L_n(g,\eps,\gamma),
\end{equation}
the \textit{asymptotically exact constants}
\begin{equation}\label{AEXdef}
\aex(g, \eps, \gamma) = \limsup_{\ell \rightarrow 0} \sup_{n\in\N,\, F\in\mathcal{F}\colon L_n(g,\eps,\gamma) = \ell} \Delta_n(F) / \ell,
\end{equation}
the \textit{lower asymptotically exact constants}
\begin{equation}\label{LowAEXdef}
\lowaex(g, \eps, \gamma) = \limsup_{\ell \rightarrow 0} \limsup_{n \rightarrow \infty} \sup_{F\in\mathcal{F}\colon L_n(g,\eps,\gamma)= \ell} \Delta_n(F) / \ell,
\end{equation}
the \textit{conditional upper asymptotically exact constants}
\begin{equation}\label{CondUpAEXdef}
\upaexcond(g, \eps, \gamma) = \sup_{\ell > 0} \limsup_{n \rightarrow \infty} \sup_{F\in\mathcal{F}\colon L_n(g,\eps,\gamma)= \ell} \Delta_n(F) / \ell.
\end{equation}
In order not to introduce excess indexes, we use identical notation for the asymptotic constants in~\eqref{AW_ineq}, \eqref{Rozovsky_ineq}, in what follows every time specifying the inequality in question. Recall that the fractions $L_n$ appearing in~\eqref{ABEdef}--\eqref{CondUpAEXdef}  depend also on the common d.f. $F$ of the random summands $X_1,\ldots,X_n$. All the introduced constants, except~\eqref{ABEdef}, assume double array scheme. The values of $L_n$ are infinitesimal in~\eqref{ABEdef}, \eqref{AEXdef}, and~\eqref{LowAEXdef}. The distinction between~\eqref{AEXdef} and~\eqref{LowAEXdef} is in the upper bound and the limit with respect to $n$, so that $\aex(g, \eps, \gamma) \ge\lowaex(g, \eps, \gamma) $, where the strict inequality may also take place, as it happens indeed, for example,  with the similar constants in the classical Berry--Esseen inequality~\cite{Chistyakov1996,Chistyakov2001a,Chistyakov2001b,Chistyakov2002,Shevtsova2010DAN,Shevtsova2010TVP}. The constants in the classical Berry--Esseen inequality similar to those defined in~\eqref{ABEdef}--\eqref{CondUpAEXdef} were firstly considered in~\cite{Esseen1956}  for~\eqref{ABEdef}, \cite{Kolmogorov1953} for~\eqref{AEXdef}, \cite{IbragimovLinnik1965,Zolotarev1986} for~\eqref{UpAEXdef},
\cite{Shevtsova2010DAN} for~\eqref{LowAEXdef}, and~\cite{Shevtsova2014DAN} for~\eqref{CondUpAEXdef}. The upper asymptotically exact~\eqref{UpAEXdef} and the conditional upper asymptotically exact~\eqref{CondUpAEXdef} constants are linked by the following relation $\upaexcond\le \upaex$ by definition, and we shall construct lower bounds namely for $\upaexcond$. As for the constants $\upaex$, we introduce them here to pay tribute to the classical works~\cite{IbragimovLinnik1965,Zolotarev1986}. The function~$g$ in~\eqref{ABEdef}, \eqref{UpAEXdef}, \eqref{CondUpAEXdef} may be arbitrary from the class~$\G$, while the constants~$\abe(g, \eps, \gamma)$ and $\lowaex(g, \eps, \gamma)$ (see~\eqref{AEXdef} and \eqref{LowAEXdef}) are defined not for all $g\in\G$. For example, $\abe(g_1, \eps, \gamma)$ and $\lowaex(g_1, \eps, \gamma)$ are not defined in any of the inequalities~\eqref{AW_ineq}, \eqref{Rozovsky_ineq}, since the corresponding fractions $\es(g_1,\eps,\gamma)$, $\roz(g_1,\eps,\gamma)$ are bounded from below by one uniformly with respect to~$\eps$ and~$\gamma$ (see~\eqref{g_1_Esseen} and~\eqref{g_1_Rozovsky}) and, hence, cannot be infinitesimal.

%\textit{Данная работа посвящена нахождению нижних оценок для правильных констант $\CE(\eps, \gamma), \CR(\eps, \gamma)$, а также асимптотических констант, определенных в~\eqref{C_best}--\eqref{C_exact_cond}. 
%}

Finally, note that definitions~\eqref{ABEdef}--\eqref{CondUpAEXdef} immediately yield the relations
\begin{equation}\label{asypt_const_in}
\max\{\lowaex,\abe\}\le \min\{\aex , \upaexcond\},\quad \max\{\aex,\upaexcond\} \le \upaex
\end{equation}
%для тех констант, которые определены 
where we omitted the arguments $g, \eps, \gamma$ for clarity).

\section{Two-point distributions}

Most of the lower bounds will be obtained by the choice of a two-point distribution 
\begin{equation}\label{two_point_dist}
\Prob\left (X_k = \sqrt{\frac{q}{p}}\right ) = 1-\Prob\left (X_k = -\sqrt{\frac{p}{q}}\right) = p, \quad q=1-p\in (0,1),\quad k=1,\ldots,n,
\end{equation}
for the random summands. The present section contains the corresponding required results. In particular, we will find values of the fractions $\es(g,\eps,\gamma)$, $\roz(g,\eps,\gamma)$ with $g=g_*,\,g_{\textsc{c}},\,g_0,\,g_1$ for all $\eps,\gamma>0$ (theorem~\ref{ThFractionsFor2PointDistr}). We will also investigate the uniform distance between the d.f. of~\eqref{two_point_dist} and the standard normal d.f. $\Phi$ and find the corresponding extreme values of the argument of d.f.'s (see theorem~\ref{ThUniDist(2point,Phi)n=1}):
\begin{equation}
\Delta_1(p) := \sup_{x \in \R}\abs{\Prob(X_1 < x) - \Phi(x)} =
\begin{cases}
\Phi(\sqrt{p/q}) - p, \quad 0 < p < \frac{1}{2}, 
\\
\Phi(\sqrt{q/p}) - q, \quad \frac{1}{2} \le p < 1.
\end{cases}
\end{equation}

\subsection{Computation of the fractions}
 
For distribution~\eqref{two_point_dist} we have $\E X_k = 0$, $\E X_k^2 = 1$, $k=1,\ldots,n$, $B_n^2=n$,
$$
g_0(z) = \min\{z, \sqrt{n}\}, \quad g_1(z) = \max\{z, \sqrt{n}\},\quad g_*(z)=z,\quad g_{\textsc{c}}(z)=1,\quad z\ge0.
$$  
Recall that for $\eps\le1$
$$
\es(g_0,\eps,\cdot)=\es(g_*,\eps,\cdot)=\es(\eps,\cdot),\quad \es(g_1,\eps,\cdot)=\es(g_{\textsc{c}},\eps,\cdot),
$$
$$
\roz(g_0,\eps,\cdot)=\roz(g_*,\eps,\cdot)=\roz(\eps,\cdot),\quad \roz(g_1,\eps,\cdot)=\roz(g_{\textsc{c}},\eps,\cdot).
$$
That is why in the formulation of the next theorem~\ref{ThFractionsFor2PointDistr}, we do not indicate values of  $\es(g_0,\eps,\cdot)$, $\es(g_1,\eps,\cdot)$, $\roz(g_0,\eps,\cdot)$, $\roz(g_1,\eps,\cdot)$ for $\eps\le1$ separately.

\begin{Theorem}\label{ThFractionsFor2PointDistr}
{\rm (i)} For all $n\in\N$, $p\in[1/2,1)$ and $\eps,\gamma>0$  we have
\begin{equation} \label{Esseen_g0_two_point}
\es(g_*,\eps,\gamma)%=\es(\eps,\gamma) 
=\begin{cases}
\eps, \quad n\eps^2\le\frac{q}{p}, 
\\
\max\left\{ \sqrt{\frac{q}{np}}, \gamma \sqrt{\frac{q^3}{np}} + \eps p \right\}, \quad {\frac{q}{p}}<n\eps^2\le\frac{p}{q}, 
\\
\max\left\{q, (\gamma q^2 + p^2)\I(p>1/2),\gamma(p-q)\right\}/\sqrt{npq}, \quad n\eps^2>\frac{p}{q},
\end{cases}
\end{equation}
$$
\es(g_0,\eps,\gamma)=\es(g_*,\eps\wedge1,\gamma),
$$
\begin{equation}\label{Esseen_g1_two_point}
\es(g_{\textsc{c}},\eps,\gamma)
= \begin{cases}
1, & n\eps^2\le{\frac{q}{p}},
\\
\max\left\{ 1, \gamma q + p \right\}, & \frac{q}{p}<n\eps^2\le\frac{p}{q},
\\
\max\left\{1,  (\gamma q + p)\I(p>\frac12), \frac{\gamma(p - q)}{p}  \right\}, & n\eps^2>\frac{p}{q},
\end{cases}
\end{equation}
\begin{equation}\label{Rozovsky_g0_two_point}
\roz(g_*,\eps,\gamma)
=\begin{cases}
\eps, & n\eps^2\le\frac{q}{p}, 
\\
\gamma\sqrt{\frac{q^3}{np}}+\max\left\{ \sqrt{\frac{q}{np}}, \eps p \right\}, & \frac{q}{p}<n\eps^2\le\frac{p}{q}, 
\\
\frac{\gamma(p-q)}{\sqrt{npq}}+ \max\left\{ \sqrt{\frac{q}{np}}, \sqrt{\frac{p^3}{nq}} \right\}, & n\eps^2>\frac{p}{q},
\end{cases}
\end{equation}
\begin{equation}\label{Rozovsky_g1_two_point}
\roz(g_{\textsc{c}},\eps,\gamma) =
\begin{cases}
1, & n\eps^2\le\frac{q}{p}, 
\\
\frac{\gamma}{\eps}\sqrt{\frac{q^3}{np}}+1, & \frac{q}{p}<n\eps^2\le\frac{p}{q}, 
\\
\frac{\gamma(p-q)}{\eps\sqrt{npq}}+1, & n\eps^2>\frac{p}{q},
\end{cases}
\end{equation}

{\rm (ii)} For all $\eps>1,\ \gamma>0,\ n\in\N$ 
$$
\es(g_1,\eps,\gamma) =1,\quad\text{if } p=1/2, 
$$
$$
\es(g_1,\eps,\gamma) = 
\begin{cases}
\max\big\{1, \gamma q + p, \frac{\gamma q^{3/2}}{\sqrt{np}}+p\eps \big\}, & n\eps^2\le \frac pq,
\\
\max\big\{1, \gamma q + p, \frac{\gamma q^2+p^2}{\sqrt{npq}}, \frac{\gamma(p - q)}{\sqrt{npq}}\big\}, & n\le \frac pq<n\eps^2,
\\
\max\big\{1, \gamma q + p, \frac{\gamma(p-q)}{p}\big\}, & n> \frac pq,
\end{cases}
\quad\text{if } p>1/2.
$$

{\rm (iii)} For all $\eps>1$, $\gamma>0$, $n\in\N$, and $p\in[1/2,1)$
$$
\roz(g_0,\eps,\gamma)
=\begin{cases}
\frac{\gamma}{\eps}\sqrt{\frac{q^3}{np}}+1
,& n\le\frac{q}{p}<n\eps^2\le\frac{p}{q}, 
\\
\frac{\gamma(p-q)}{\eps\sqrt{npq}}+1
,& n\le\frac{q}{p}\le\frac{p}{q}<n\eps^2, 
\\
\frac{\gamma}{\eps}\sqrt{\frac{q^3}{np}}
+\max\left\{ \sqrt{\frac{q}{np}}, p \right\}
,& \frac{q}{p}<n<n\eps^2\le\frac{p}{q}, 
\\
\frac{\gamma(p-q)}{\eps\sqrt{npq}}
+\max\left\{ \sqrt{\frac{q}{np}}, p \right\}
,& \frac{q}{p}<n\le\frac{p}{q}<n\eps^2, 
\\
\frac{\gamma(p-q)}{\eps\sqrt{npq}}
+\max\left\{ \sqrt{\frac{q}{np}}, \sqrt{\frac{p^3}{nq}}\right\}
,& n>\frac{p}{q},
\end{cases}
$$
$$
\roz(g_1,\eps,\gamma) =
\begin{cases}
\gamma\sqrt{\frac{q^3}{np}}+\max\left\{ \sqrt{\frac{q}{np}}, \eps p\right\}, 
& n\le\frac{q}{p}<n\eps^2\le\frac{p}{q}, 
\\
\frac{\gamma(p-q)}{\sqrt{npq}}+\max\left\{ \sqrt{\frac{q}{np}}, \sqrt{\frac{p^3}{nq}}\right\}, 
& n\le\frac{q}{p}\le\frac{p}{q}<n\eps^2, 
\\
\gamma\sqrt{\frac{q^3}{np}}+\max\left\{ \eps p,1\right\},
& \frac{q}{p}<n<n\eps^2\le\frac{p}{q}, 
\\
\frac{\gamma(p-q)}{\sqrt{npq}}+\max\left\{ \sqrt{\frac{p^3}{nq}},1\right\},
& \frac{q}{p}<n\le\frac{p}{q}<n\eps^2, 
\\
\frac{\gamma(p-q)}{\sqrt{npq}}+1,
& \frac{p}{q}<n.
\end{cases}
$$
In particular,
$$
\roz(g_{\textsc{c}},\infty,\,\cdot\,)\equiv1,\quad p\in\big[\tfrac12,1\big),\ n\in\N,
$$
for $p=1/2$ and all $\eps,\gamma>0,\ n\in\N:$
$$
\es(g_*,\eps,\gamma)=\es(g_0,\eps,\gamma) =\roz(g_*,\eps,\gamma)=\roz(g_0,\eps,\gamma) =\min\big\{\eps,\tfrac1{\sqrt n}\big\},
$$
$$
\es(g_{\textsc{c}},\eps,\gamma)=\es(g_1,\eps,\gamma)=
\roz(g_{\textsc{c}},\eps,\gamma)=\roz(g_1,\eps,\gamma)\equiv1,
$$
for $p>1/2$ and all $n\in\N:$
$$
\es(g_*,1,1)=\es(g_0,1,1)=\es(g_0,\infty,1) = %\es(g_0,1,1)=\es(1,1)=
\begin{cases}
\frac{q^2}{\sqrt{npq}} + p, & n\le\frac{p}{q}, 
\\
\frac{q^2+p^2}{\sqrt{npq}}, & n>\frac{p}{q},
\end{cases}
$$
$$
\es(g_*,\infty,1)=%\es(\infty,1)=
\frac{q^2+p^2}{\sqrt{npq}},
\qquad\qquad\qquad\qquad \qquad\qquad \quad\
$$
%$$
%\es(g_*,\infty,1)=\es(\infty,1)=\frac{q^2+p^2}{\sqrt{npq}},
%$$
$$
\es(g_0,\infty,1) = 
\begin{cases}
\max\left\{ \sqrt{\frac{q}{np}},\sqrt{\frac{q^3}{np}} + p, \frac{p-q}p\right\}, & n\le \frac pq,
\\
\frac{q^2 + p^2}{\sqrt{npq}}, & n> \frac pq,
\end{cases}
$$
$$
\es(g_{\textsc{c}},\eps,\gamma)=1,\quad \eps>0,\ \gamma\le1. \qquad\qquad\qquad\qquad 
$$
%$$
%\es(g_1,\eps,1) = 
%\begin{cases}
%\max\big\{1, \frac{q^{3/2}}{\sqrt{np}}+p\eps \big\}, \quad n\eps^2\le \frac pq,
%\\
%\max\left \{1, \frac{q^2+p^2}{\sqrt{npq}}\right \}, \quad n\le \frac pq<n\eps^2,
%\\
%1, \quad n>\frac pq,
%\end{cases}
%\eps>1,
%$$
\end{Theorem}

%$L_{\bullet\,,\,n}\in\{\es,\roz\}$
\begin{proof} 
Observe that in the i.i.d. case under consideration we have $B_n^2=n$,
$$
\es(g,\eps, \gamma) = \frac{1}{g(\sqrt{n})}  \sup_{0 < z < \eps\sqrt{n}} \frac{g(z)}{z} \left \{ \gamma \abs{\mu_1(z)} + z\sigma_1^2(z) \right \},
$$
$$
\roz(g,\eps, \gamma) =\frac{1}{g(\sqrt{n})}\bigg(\frac{g(\eps\sqrt{n})}{\eps\sqrt{n}}\, \gamma\abs{\mu_1(\eps\sqrt{n})}+\sup_{0<z < \eps\sqrt{n}} g(z)\sigma_1^2(z)\bigg).
$$
Let us find $\mu_1(\,\cdot\,)$, $\sigma_1^2(\,\cdot\,)$. For all $z>0$ we have
$$
\abs{\mu_1(z)} = \abs{\E X_1^3\I(|X_1| < z)}= 
\begin{cases}
0,& z \le \sqrt{\frac{q}{p}},
\\
\sqrt{\frac{q^3}{p}}, &\sqrt{\frac{q}{p}} < z \le \sqrt{\frac{p}{q}},
\\
\frac{p - q}{\sqrt{pq}},& z > \sqrt{\frac{p}{q}},
\end{cases}
$$
$$
\sigma_1^2(z) = \E X_1^2\I(|X_1|\ge z)= 
\begin{cases}
1,& z\le\sqrt{\frac{q}{p}}, \\
p, &\sqrt{\frac{q}{p}}<z\le\sqrt{\frac{p}{q}}, \\
0, &z > \sqrt{\frac{p}{q}},
\end{cases}
$$
$$
\gamma \abs{\mu_1(z)} + z\sigma_1^2(z) = 
\begin{cases}
z, &z \le \sqrt{\frac{q}{p}}, \\
\gamma \sqrt{\frac{q^3}{p}} + zp, &\sqrt{\frac{q}{p}} < z \le \sqrt{\frac{p}{q}}, \\
\frac{\gamma(p - q)}{\sqrt{pq}}, &z > \sqrt{\frac{p}{q}}.
\end{cases}
$$

1) Compute $\es(g_*,\eps,\gamma)$ and $\es(g_0,\eps,\gamma)$. For all $n\in\N$, $\eps,\gamma>0$ we have
$$
\es(g_*,\eps,\gamma)%=\es(\eps,\gamma)
=\frac1{\sqrt{n}}\sup_{0<z<\eps\sqrt{n}}\left\{ \gamma\abs{\mu_1(z)}+z\sigma_1^2(z)\right \}=
$$
$$
=\begin{cases}
\eps, & n\eps^2\le\frac{q}{p}, 
\\
\max\left\{ \sqrt{\frac{q}{np}}, \gamma \sqrt{\frac{q^3}{np}} + p\eps \right\}, & {\frac{q}{p}}<n\eps^2\le\frac{p}{q}, 
\\
\max\left\{ \sqrt{\frac{q}{np}}, \Big(\gamma \sqrt{\frac{q^3}{np}} + \sqrt{\frac{p^3}{nq}}\Big)\I(p>\frac12), \frac{\gamma(p-q)}{\sqrt{npq}}\right\}, & n\eps^2>\frac{p}{q},
\end{cases}
$$
%что совпадает с $\es(g_0,\eps,\gamma)$ при $\eps\le1$, 
in particular, $\es(g_*,\eps,\gamma)=\min\big\{\eps,\frac1{\sqrt n}\big\}$ for $p=\frac12$. Now let us find $\es(g_0,\eps,\gamma) $ for $\eps>1$. We have
\begin{multline*}
\es(g_0,\eps,\gamma) =\sup_{0 < z < \eps\sqrt{n}} \frac{z\wedge\sqrt{n}}{z\sqrt{n}} \left \{ \gamma \abs{\mu_1(z)} + z\sigma_1^2(z) \right \}=
\\
=\max\bigg\{
\frac1{\sqrt{n}}\sup_{0 < z < \sqrt{n}}\left \{ \gamma \abs{\mu_1(z)} + z\sigma_1^2(z) \right \},
\sup_{\sqrt{n}\le  z <\eps\sqrt{n}}  \frac{1}{z} \left \{ \gamma \abs{\mu_1(z)} + z\sigma_1^2(z) \right \}\bigg \}=
\\
=\max\bigg \{ \es(g_*,1,\gamma),
\sup_{\sqrt{n}\le  z <\eps\sqrt{n}} \left \{ \frac{\gamma}z \abs{\mu_1(z)} + \sigma_1^2(z) \right \}\bigg \},\quad \gamma>0.
\end{multline*}
If $p=q=1/2$, then $\es(g_*,\eps,\gamma)=1/{\sqrt n}$,
$$
\sup_{\sqrt{n}\le  z <\eps\sqrt{n}} \left \{ \frac{\gamma}z \abs{\mu_1(z)} + \sigma_1^2(z) \right \} = \sup_{\sqrt{n}\le  z <\eps\sqrt{n}}\I(z\le1)=\I(n=1)
$$
and hence,
$$
\es(g_0,\eps,\gamma) =\max\big\{\tfrac1{\sqrt n},\I(n=1)\big\} =\tfrac1{\sqrt n},\quad \eps>1,
$$
so that we may write for all $\eps>0$
$$
\es(g_0,\eps,\,\cdot\,)=\min\big\{\eps,\tfrac1{\sqrt n}\big\} =\es(g_*,\eps,\,\cdot\,)=\es(g_*,\eps\wedge1,\,\cdot\,).
$$
If $p>1/2$, then
$$
\sup_{\sqrt{n}\le  z <\eps\sqrt{n}} \left \{ \frac{\gamma}z \abs{\mu_1(z)} + \sigma_1^2(z) \right \} =\sup_{\sqrt{n}\le  z <\eps\sqrt{n}} 
\begin{cases}
\frac{\gamma}z \sqrt{\frac{q^3}{p}} + p, & z \le \sqrt{\frac{p}{q}}, 
\\
\frac{\gamma(p-q)}{z\sqrt{pq}}, & z > \sqrt{\frac{p}{q}},
\end{cases}= 
$$
$$
=
\begin{cases}
\gamma\sqrt{\frac{q^3}{np}} + p, &   n\eps^2\le \frac pq,
\\
\max\left\{\gamma\sqrt{\frac{q^3}{np}} + p,\frac{\gamma(p-q)}p\right\}, & n\le\frac pq<n\eps^2,
\\
\frac{\gamma(p - q)}{\sqrt{npq}}, & n>\frac pq,
\end{cases}
$$
and with the account of $p\le\sqrt{\frac{p^3}{nq}}$, $\frac{\gamma(p-q)}{p}\le \frac{\gamma(p-q)}{\sqrt{npq}}$ for $n\le\frac pq$,we finally obtain
$$
\es(g_0,\eps,\gamma) = \es(g_*,1,\gamma) = 
\begin{cases}
\max\left\{ \sqrt{\frac{q}{np}}, \gamma \sqrt{\frac{q^3}{np}} + p\right\}, & n\eps^2\le \frac pq,
\\
%\max\left\{
%\sqrt{\frac{q}{np}},
%\gamma\sqrt{\frac{q^3}{np}} + \sqrt{\frac{p^3}{nq}},
%\frac{\gamma(p-q)}{\sqrt{npq}},
%\gamma\sqrt{\frac{q^3}{np}} + p,
%\frac{\gamma(p-q)}p
%\right\} =
\max\left\{
\sqrt{\frac{q}{np}},
\gamma\sqrt{\frac{q^3}{np}} + \sqrt{\frac{p^3}{nq}},
\frac{\gamma(p-q)}{\sqrt{npq}}
\right\},& n\eps^2>\frac pq,
\end{cases}
$$
for all $\eps>1$ and $\gamma>0$, that is, for all $\eps>0$ the identity
$$
\es(g_0,\eps,\,\cdot\,) = \es(g_*,\eps\wedge1,\,\cdot\,) 
$$
holds also for $p>\frac12$. In particular, with $\gamma=1$ and $p>\frac12$ for all $\eps>0$ we have
$$
\es(g_*,\eps,1)=%\es(\infty,1)=
\begin{cases}
\eps, & n\eps^2\le\frac{q}{p}, 
\\
\max\left\{ \sqrt{\frac{q}{np}}, \sqrt{\frac{q^3}{np}} + p\eps \right\} =\sqrt{\frac{q^3}{np}} + p\eps, 
& {\frac{q}{p}}<n\eps^2\le\frac{p}{q}, 
\\
\max\left\{q, q^2 + p^2,p-q\right\}/\sqrt{npq} =\frac{q^2+p^2}{\sqrt{npq}}, 
& n\eps^2>\frac{p}{q},
\end{cases}
$$
$$
\es(g_*,\infty,1)=%\es(\infty,1)=
\frac{q^2+p^2}{\sqrt{npq}},
$$
$$
\es(g_*,1,1)=\es(g_0,1,1)=\es(g_0,\infty,1) = 
\begin{cases}
\sqrt{\frac{q^3}{np}} + p, & n\le\frac{p}{q}, 
\\[2mm]
\frac{q^2+p^2}{\sqrt{npq}}, & n>\frac{p}{q}.
\end{cases}
$$
%$\es(g_0,\eps,1) = \es(g_*,\eps,1),$ $\eps\le1$, а при $\eps>1$
%$$
%\es(g_0,\eps,1) = 
%\begin{cases}
%\max\left\{ \sqrt{\frac{q}{np}}, \sqrt{\frac{q^3}{np}} + p\right\}, \quad n\eps^2\le \frac pq,
%\\
%\max\left\{ \sqrt{\frac{q}{np}},\sqrt{\frac{q^3}{np}} + p, \frac{p-q}p\right\}, \quad n\le \frac pq<n\eps^2,
%\\
%\max\left\{q^2 + p^2\right\}/\sqrt{npq}, \quad n> \frac pq,
%\end{cases}
%$$
%в частности, $\es(g_*,\eps,1)= (p^2+q^2)/\sqrt{npq}$ при $n\eps^2\ge p/q$ (также см.~\cite{GabdullinMakarenkoShevtsova2018}).

2) Let us find  $\es(g_{\textsc{c}},\eps,\gamma)$ and $\es(g_1,\eps,\gamma)$. For all $n\in\N$, $\eps,\,\gamma>0$ we have
\begin{multline*}
\es(g_{\textsc{c}},\eps,\gamma) =\sup_{0<z<\eps\sqrt{n}}\left\{\frac{\gamma}z \abs{\mu_1(z)} + \sigma_1^2(z)\right\}
 = \sup_{0<z<\eps\sqrt{n}}\begin{cases}
1, & z \le \sqrt{\frac{q}{p}},
\\
\frac{\gamma}{z}  \sqrt{\frac{q^3}{p}} + p, & \sqrt{\frac{q}{p}} < z \le \sqrt{\frac{p}{q}},
\\
\frac{\gamma(p - q)}{z\sqrt{pq}}, & z > \sqrt{\frac{p}{q}},
\end{cases}=
\\
= \begin{cases}
1, & n\eps^2\le{\frac{q}{p}},
\\
\max\left\{ 1, \gamma q + p \right\}, & \frac{q}{p}<n\eps^2\le\frac{p}{q},
\\
\max\left\{1, (\gamma q + p)\I(p>\frac12), \frac{\gamma(p - q)}{p}  \right\}, & n\eps^2>\frac{p}{q},
\end{cases}
\end{multline*}
in particular, $\es(g_{\textsc{c}},\eps,\gamma)\equiv1$ for $p=\tfrac12$ and all $\eps,\gamma>0,$ as well as for $p>\tfrac12$, $\eps>0,\ \gamma\le1.$  For $\eps>1$ we have
\begin{multline*}
\sqrt{n}\es(g_1,\eps,\gamma) =\sup_{0 < z < \eps\sqrt{n}} \frac{z\vee\sqrt{n}}{z} \left \{ \gamma \abs{\mu_1(z)} + z\sigma_1^2(z) \right \}=
\\
=\max\bigg\{\sup_{0 < z < \sqrt{n}} \frac{\sqrt{n}}{z} \left \{ \gamma \abs{\mu_1(z)} + z\sigma_1^2(z) \right \}, \sup_{\sqrt{n}\le  z <\eps\sqrt{n}} \left \{ \gamma \abs{\mu_1(z)} + z\sigma_1^2(z) \right \}\bigg \}=
\\
=\max\bigg \{ \sqrt{n}\es(g_{\textsc{c}},1,\gamma),
\sup_{\sqrt{n}\le  z <\eps\sqrt{n}} \left \{ \gamma \abs{\mu_1(z)} + z\sigma_1^2(z) \right \}\bigg \},\quad \gamma>0.
\end{multline*}
If $p=q=1/2$, then %$\sqrt{n}\es(g_{\textsc{c}},1,\gamma)=\sqrt n$,
$$
\sup_{\sqrt{n}\le  z <\eps\sqrt{n}} \left \{ \gamma \abs{\mu_1(z)} + z\sigma_1^2(z) \right \}= \sup_{\sqrt{n}\le  z <\eps\sqrt{n}}z\I(z\le1) =\I(n=1),
$$
and hence,
$$
\es(g_1,\eps,\gamma) = \max\{1,\I(n=1)/\sqrt n\}=1,\quad \eps>1,\gamma>0,\ n\in\N;
$$
otherwise
$$
\sup_{\sqrt{n}\le  z <\eps\sqrt{n}} \left \{ \gamma \abs{\mu_1(z)} + z\sigma_1^2(z) \right \}
=
$$
$$
=
\sup_{\sqrt{n}\le  z <\eps\sqrt{n}}
\begin{cases}
\gamma \sqrt{\frac{q^3}{p}} + zp, &  z \le \sqrt{\frac{p}{q}},
\\
\frac{\gamma(p - q)}{\sqrt{pq}}, & z > \sqrt{\frac{p}{q}},
\end{cases}
=
\begin{cases}
{\gamma q^{3/2}}/{\sqrt{p}}+p\eps\sqrt{n}, &   n\eps^2\le \frac pq,
\\
{\max\{\gamma q^2+p^2,\gamma(p-q)\}}/{\sqrt{pq}}, & n\le\frac pq<n\eps^2,
\\
{\gamma(p - q)}/{\sqrt{pq}}, & n>\frac pq,
\end{cases}
$$
and hence,
$$
\es(g_1,\eps,\gamma) = 
\begin{cases}
\max\big\{1, \gamma q + p, \frac{\gamma q^{3/2}}{\sqrt{np}}+p\eps \big\}, \quad n\eps^2\le \frac pq,
\\
\max\big\{1, \gamma q + p, {\max\{\gamma q^2+p^2,\gamma(p - q)\}}/{\sqrt{npq}}\big\}, \quad n\le \frac pq<n\eps^2,
\\
\max\big\{1, \gamma q + p, \frac{\gamma(p-q)}{p}, \frac{\gamma(p-q)}{\sqrt{npq}}\big\}
=\max\big\{1, \gamma q + p, \frac{\gamma(p-q)}{p}\big\}, \quad n> \frac pq,
\end{cases}
$$
for all $\eps>1$ and $\gamma>0$. In particular, 
$$
\es(g_1,\eps,1) = 
\begin{cases}
\max\big\{1, \frac{q^{3/2}}{\sqrt{np}}+p\eps \big\}, & n\eps^2\le \frac pq,
\\
\max\left \{1, \frac{q^2+p^2}{\sqrt{npq}}\right \}, & n\le \frac pq<n\eps^2,
\\
1, & n>\frac pq,
\end{cases}
\quad p>\tfrac12,\ \eps>1.
$$

3) Let us compute $\roz(g_*,\eps,\gamma)$ and $\roz(g_0,\eps,\gamma)$. 
With the account of
$$
\sup_{0<z\le\eps\sqrt{n}}z\sigma_1^2(z)
=\sup_{0<z\le\eps\sqrt{n}}
\begin{cases}
z,& z\le\sqrt{\frac{q}{p}},
\\
pz, &\sqrt{\frac{q}{p}}<z\le\sqrt{\frac{p}{q}},
\\
0, &z > \sqrt{\frac{p}{q}},
\end{cases}
=\begin{cases}
\eps\sqrt n, \quad n\eps^2\le\frac{q}{p}, 
\\
\max\left\{ \sqrt{\frac{q}{p}}, p\eps\sqrt n \right\}, \quad \frac{q}{p}<n\eps^2\le\frac{p}{q}, 
\\
\max\left\{ \sqrt{\frac{q}{p}}, \sqrt{\frac{p^3}q}\I(p>\frac12) \right\}, \quad n\eps^2>\frac{p}{q},
\end{cases}
$$
and $\max\left\{ \sqrt{\frac{q}{p}}, \sqrt{\frac{p^3}q}\I(p>\frac12) \right\}=\max\left\{ \sqrt{\frac{q}{p}}, \sqrt{\frac{p^3}q}\right\}$, $p\in[\frac12,1)$, for all $n\in\N$, $\eps,\gamma>0$ we obtain
$$
\roz(g_*,\eps,\gamma)
=\frac1{\sqrt{n}}\Big(\gamma\abs{\mu_1(\eps\sqrt{n})} +\sup_{0<z<\eps\sqrt{n}}z\sigma_1^2(z)\Big)=
$$
$$
=\begin{cases}
\eps, & n\eps^2\le\frac{q}{p}, 
\\
\gamma\sqrt{\frac{q^3}{np}}+\max\left\{ \sqrt{\frac{q}{np}}, p\eps \right\}, & \frac{q}{p}<n\eps^2\le\frac{p}{q}, 
\\
\frac{\gamma(p-q)}{\sqrt{npq}} +\max\left\{ \sqrt{\frac{q}{np}}, \sqrt{\frac{p^3}{nq}}\right\}, & n\eps^2>\frac{p}{q},
\end{cases}
$$
in particular, $\roz(g_*,\eps,\gamma)=\min\{\eps,\frac1{\sqrt n}\}$ for $p=\frac12$. Now let $\eps>1$. Taking into account that $\sigma_1^2(z)$ does not increase with respect to $z\ge0$, and using the just computed $\sup\limits_{0<z\le\eps\sqrt{n}}z\sigma_1^2(z)$, we obtain
$$
\roz(g_0,\eps,\gamma)
=\frac1{\sqrt{n}}\Big(\frac{\gamma}{\eps}\abs{\mu_1(\eps\sqrt{n})} +\sup_{0<z<\eps\sqrt{n}}\min\{z,\sqrt n\}\sigma_1^2(z)\Big)=
$$
$$
=\frac1{\sqrt{n}}\Big(\frac{\gamma}{\eps}\abs{\mu_1(\eps\sqrt{n})} +\sup_{0<z\le \sqrt{n}}z\sigma_1^2(z)\Big)=
$$
$$
=\begin{cases}
\frac{\gamma}{\eps}\sqrt{\frac{q^3}{np}}+1
,& n\le\frac{q}{p}<n\eps^2\le\frac{p}{q}, 
\\
\frac{\gamma(p-q)}{\eps\sqrt{npq}}+1
,& n\le\frac{q}{p}\le\frac{p}{q}<n\eps^2, 
\\
\frac{\gamma}{\eps}\sqrt{\frac{q^3}{np}}
+\max\left\{ \sqrt{\frac{q}{np}}, p \right\}
,& \frac{q}{p}<n<n\eps^2\le\frac{p}{q}, 
\\
\frac{\gamma(p-q)}{\eps\sqrt{npq}}
+\max\left\{ \sqrt{\frac{q}{np}}, p \right\}
,& \frac{q}{p}<n\le\frac{p}{q}<n\eps^2, 
\\
\frac{\gamma(p-q)}{\eps\sqrt{npq}}
+\max\left\{ \sqrt{\frac{q}{np}}, \sqrt{\frac{p^3}{nq}}%\I(p>\frac12)
\right\}
,& n>\frac{p}{q},
\end{cases}
$$
in particular, for $p=\frac12$ we have $\roz(g_0,\eps,\,\cdot\,)=\frac1{\sqrt n} =\min\{\eps,\frac1{\sqrt n}\}=\roz(g_*,\eps,\,\cdot\,)$ for all $\eps>1$, and hence, for all $\eps>0$.

4) Let us compute $\roz(g_{\textsc{c}},\eps,\gamma)$ and $\roz(g_1,\eps,\gamma)$. Taking into account that $\sigma_1^2(z)$ does not increase with respect to $z\ge0$ and $\sigma_1^2(0)=\sigma_1^2=1$, for all $n\in\N$, $\eps,\gamma>0$ and $p\in[\frac12,1)$ we obtain
$$
\roz(g_{\textsc{c}},\eps,\gamma) =\frac{\gamma}{\eps\sqrt{n}}\abs{\mu_1(\eps\sqrt{n})} +\sup_{0<z<\eps\sqrt{n}}\sigma_1^2(z)= \frac{\gamma}{\eps\sqrt{n}}\abs{\mu_1(\eps\sqrt{n})}+1=
$$
$$
=\begin{cases}
1, & n\eps^2\le\frac{q}{p}, 
\\
\frac{\gamma}{\eps}\sqrt{\frac{q^3}{np}}+1, & \frac{q}{p}<n\eps^2\le\frac{p}{q}, 
\\
\frac{\gamma(p-q)}{\eps\sqrt{npq}}+1, & n\eps^2>\frac{p}{q},
\end{cases}
$$
in particular, $\roz(g_{\textsc{c}},\,\cdot\,,\,\cdot\,)\equiv1$ for $p=\frac12$ and $\roz(g_{\textsc{c}},\infty,\,\cdot\,)\equiv1$ for all $p\in[\frac12,1)$. Now let $\eps>1$. With the account of
\begin{multline*}
\sup_{\sqrt{n}\le z<\eps\sqrt{n}}z\sigma_1^2(z)= 
\\
=\sup_{\sqrt{n}\le z<\eps\sqrt{n}}
\begin{cases}
z,& z\le\sqrt{\frac{q}{p}},
\\
pz, &\sqrt{\frac{q}{p}}<z\le\sqrt{\frac{p}{q}},
\\
0, &z > \sqrt{\frac{p}{q}},
\end{cases}
=\begin{cases}
\max\left\{ \sqrt{\frac{q}{p}}, p\sqrt{n\eps^2\wedge \frac pq} \right\}, & n\le\frac{q}{p}<n\eps^2, 
\\
p\sqrt{n\eps^2\wedge \frac pq}, & \frac{q}{p}<n\le \frac pq, 
\\
0,&n>\frac pq,
\end{cases}
\end{multline*}
we have
$$
\sqrt{n}\roz(g_1,\eps,\gamma) =\gamma\abs{\mu_1(\eps\sqrt{n})} +\sup_{0<z<\eps\sqrt{n}}(z\vee\sqrt{n})\sigma_1^2(z)=
$$
$$
=\gamma\abs{\mu_1(\eps\sqrt{n})} +\max\Big\{\sqrt{n},\sup_{\sqrt{n}\le z<\eps\sqrt{n}} z\sigma_1^2(z) \Big\}=
$$
$$
=\begin{cases}
\gamma\sqrt{\frac{q^3}p}+\max\left\{ \sqrt{\frac{q}{p}}, p\eps\sqrt{n},\sqrt n\right\} 
=\gamma\sqrt{\frac{q^3}p}+\max\left\{ \sqrt{\frac{q}{p}}, p\eps\sqrt{n}\right\}, 
& n\le\frac{q}{p}<n\eps^2\le\frac{p}{q}, 
\\
\frac{\gamma(p-q)}{\sqrt{pq}}+\max\left\{ \sqrt{\frac{q}{p}}, \sqrt{\frac{p^3}q},\sqrt n\right\}
=\frac{\gamma(p-q)}{\sqrt{pq}}+\max\left\{ \sqrt{\frac{q}{p}}, \sqrt{\frac{p^3}q}\right\}, 
& n\le\frac{q}{p}\le\frac{p}{q}<n\eps^2, 
\\
\gamma\sqrt{\frac{q^3}p}+\max\left\{ p\eps\sqrt{n}, \sqrt{n}\right\},
& \frac{q}{p}<n<n\eps^2\le\frac{p}{q}, 
\\
\frac{\gamma(p-q)}{\sqrt{pq}}+\max\left\{ \sqrt{\frac{p^3}q}, \sqrt{n}\right\},
& \frac{q}{p}<n\le\frac{p}{q}<n\eps^2, 
\\
\frac{\gamma(p-q)}{\sqrt{pq}}+\sqrt{n},
& n>\frac{p}{q},
\end{cases}
$$
in particular, $\roz(g_1,\eps,\gamma)\equiv 1$ for $p=\frac12$.
\end{proof}

\subsection{Computation of the uniform distance}

In the present section we compute the uniform distance  $\Delta_1(p)$
between the d.f. of~\eqref{two_point_dist} and the standard normal d.f. Let us denote
\begin{equation}\label{Psi_func}
\Psi(x) = \frac{1}{1 + x^2} - \Phi(-\abs{x}) = \Phi(\abs{x}) - \frac{x^2}{1+x^2}, \quad x \in \R,
\end{equation}
\begin{equation}\label{PsiWave_func}
\wave{\Psi}(p) = \Psi\left(\sqrt{\frac{1-p}{p}}\right) = \Phi\left(\sqrt{\frac{1-p}{p}}\right) - (1 - p), \quad 0 < p < 1.
\end{equation}
Note that $\Psi(x)$  is an even function,  therefore, it suffices to investigate it only for $x \ge 0.$

\begin{Lemma}[see \cite{ChebotarevKondrikMikhailov2007}]\label{ChebotarevPsi} 
The function $\Psi(x)$ is positive for $x \ge 0,$ increases for $0 < x < x_{\Phi},$ decreases for $x > x_{\Phi}$ and attains its maximal value $C_{\Phi} = 0.54093\ldots$ in the point  $x_{\Phi} = 0.213105\ldots,$ where $x_{\Phi}$ is the unique root of the equation
$$
xe^{x^2/2}(1+x^2)^{-2}=(8\pi)^{-1/2}.
$$
\end{Lemma}

\begin{Remark}
The statement about  the interval of monotonicity is absent in the formulation of the corresponding lemma in~\cite{ChebotarevKondrikMikhailov2007}, but these intervals were  investigated  in the proof.
\end{Remark}

The definition of~$\wave{\Psi}$ and lemma~\ref{ChebotarevPsi} immediately imply that the function
$$
\wave{\Psi}(p) := \Phi\left(\sqrt{\frac{1-p}{p}}\right) - (1 - p)
$$
is positive for $p \in (0, 1),$ increases for $0 < p < p_{\Phi}$, decreases for $p_{\Phi} < p < 1$ and
$$
\max_{p \in (0,1)}\wave{\Psi}(p) = \wave{\Psi}(p_{\Phi}) = C_{\Phi},
\quad \text{where }\ p_{\Phi} = \frac{1}{x_{\Phi}^2 + 1} = 0.9565\ldots\ .
$$

%следующая лемма.
%\begin{Lemma}\label{PsiWaveLemma}
%Функция
%$$
%    \wave{\Psi}(p) = \Phi\left(\sqrt{\frac{1-p}{p}}\right) - (1 - p)
%$$
%положительна при $p \in (0, 1),$ возрастает при $0 < p < p_{\Phi}$, убывает при $p_{\Phi} < p < 1$ и
%$$
%    \max_{p \in (0,1)}\wave{\Psi}(p) = \wave{\Psi}(p_{\Phi}) = C_{\Phi},
%$$
%где
%$$
%p_{\Phi} = \frac{1}{x_{\Phi}^2 + 1} = 0.9565\ldots\ .
%$$
%\end{Lemma}

\begin{Lemma}\label{phiLemma}
On the interval $p \in (0,1)$ the function
$$
\phi(p):= \wave{\Psi}(1-p) + p =\Phi\left( \sqrt{ \frac{p}{1-p} } \right)
$$
monotonically increases, while its derivative
$$
\phi'(p) = \frac{e^{-\frac{p}{2(1-p)}}}{2\sqrt{2\pi}\sqrt{p(1-p)^3}}
$$
monotonically decreases.
\end{Lemma}

\begin{proof}
The function $\phi(p)$ monotonically increases as a superposition of monotonically increasing functions. Note that the derivative $\phi'(p)$ takes only positive values, hence, we may define its logarithm
   $$
    \ln\phi'(p) = -\frac{p}{2(1-p)} - \ln(2\sqrt{2\pi}) - \frac{1}{2}\ln p - \frac{3}{2} \ln (1-p) =: u(p).
    $$
The derivative
    $$
u'(p) = -\frac{1}{2(1-p)^2} - \frac{1}{2p} + \frac{3}{2(1-p)} = 0
    $$
changes its sign in the points $p\in(0,1)$ that are the roots of the equation
    $$
    -p -(1-p)^2 + 3p(1-p) = 0,
    $$
which has a unique solution $p=1/2$ on $(0,1)$.  Since
    $$
    \frac{d}{dp}u(p)\rvert_{p = \frac{1}{4}} = -\frac{8}{9}, \quad \frac{d}{dp}u(p)\rvert_{p = \frac{3}{4}} = -\frac{8}{3},
    $$
then $u(p)$ is strictly decreasing on $(0,1)$, and hence,  $\phi'(p)$ is also strictly decreasing on $(0,1).$    
\end{proof}

\begin{Theorem}\label{ThUniDist(2point,Phi)n=1}
Let $X_1$ have distribution~\eqref{two_point_dist}, then
$$
\Delta_1(p) := \sup_{x \in \R}\big|\Prob(X_1 < x) - \Phi(x)\big|= \wave{\Psi}( p \vee q ) = \begin{cases}
\Phi(\sqrt{p/q}) - p, \quad 0 < p < \frac{1}{2}, \\
\Phi(\sqrt{q/p}) - q, \quad \frac{1}{2} \le p < 1.
\end{cases}
$$
\end{Theorem}

\begin{proof}
Note that
$$
\Delta_1(p) = \max\left\{\Phi(-\sqrt{p/q}), \big|\Phi(-\sqrt{p/q}) - q\big|, \big|\Phi(\sqrt{q/p})-q\big|, 1 - \Phi(\sqrt{q/p})\right\} = 
$$ 
$$
= \max \big\{ 1 - \phi(p), \wave{\Psi}(q), \wave{\Psi}(p), 1 - \phi(q) \big\} 
= \max \big\{ f_1(p), f_2(p), f_3(p), f_4(p) \big\},
$$
where
\begin{align*}
f_1(p) &= 1 - \phi(p), \\
f_2(p) &= \wave{\Psi}(1-p) = \phi(p) - p, \\
f_3(p) &= \wave{\Psi}(p) = \phi(1 - p) - (1 - p), \\
f_4(p) &= 1 - \phi(1-p).
\end{align*}
    
Let $p \ge q.$ Then $f_1(p) \le f_4(p),$ and it suffices to show that
$$
\max\big\{f_2(p), f_4(p)\big\} \le f_3(p).
$$
Let us prove that $f_5(p):=f_3(p)-f_4(p)\ge0$. Using lemma~\ref{phiLemma}, we conclude that the derivative
$$
f_5'(p):= \frac{d}{dp}( 2\phi(1-p) + p - 2) = 2\frac{d}{dp}\phi(1-p) + 1 = 1 - \frac{e^{-\frac{(1-p)}{2p}}}{\sqrt{2\pi}\sqrt{(1-p)p^3}}
$$
decreases on  $(0,1)$ with
$$
f_5'\big (\tfrac12\big ) = 1 - \tfrac{4}{\sqrt{2\pi e}} = 0.0321\ldots,\quad \lim_{p\to1-0}f_5'(p)=-\infty,
$$
hence, $f_5(p)$ has a unique stationary point on $(0,1)$, which is a point of maximum, so that
$$
\inf_{\frac{1}{2} \le p < 1} f_5(p) = \min\Big\{ f_5\big(\tfrac{1}{2}\big) , \lim_{p\to1-}f_5(p) \Big\} = \min\left\{ 2\Phi(1) - \tfrac{3}{2}, 0 \right\} = \min\left\{0.1826\ldots, 0\right\} = 0.
$$
The inequality $f_2(p) \le f_3(p)$ follows from the properties of the function $\wave{\Psi}$ with the account of
$$
\wave{\Psi}\big(\tfrac{1}{2}\big) \wedge \wave{\Psi}(1-0) = \wave{\Psi}\big(\tfrac{1}{2}\big).
$$

Thus, the theorem has been proved in the case $p \ge q$. The validity of the statement of the theorem for  $p < q$ follows from that $f_1(p) = f_4(q)$ and $f_2(p) = f_3(q)$.
\end{proof}

\begin{Lemma}[see \cite{Agnew1957,ChebotarevKondrikMikhailov2007}] \label{ChebotarevBound}
For an arbitrary d.f. $F$ with zero mean and unit variance we have
$$
\sup_{x \in \R}\abs{F(x) - \Phi(x)} \le \sup_{x \in \R}\Psi(x) = \Psi(x_{\Phi}) =  \wave{\Psi}(p_{\Phi}) = \Delta_1(p_{\Phi}) =: C_{\Phi} = 0.54093\ldots.
$$
\end{Lemma}
\begin{Remark}
In~\cite[(2.32), (2.33)]{Agnew1957} it is proved that
$$
\abs{F(x) - \Phi(x)} \le \Psi(x), \quad x \in \R.
$$
In~\cite{ChebotarevKondrikMikhailov2007} it is proved that in the inequality
$$
\sup_{x \in \R}\abs{F(x) - \Phi(x)} \le \sup_{x \in \R}\Psi(x) = C_{\Phi}
$$
the equality is attained at a two-point distribution.
\end{Remark}

\section{Lower bounds for exact constants}
In the present section we construct lower bounds for the quantities
$$
\inf_{g \in \G} \CE(g,\eps, \gamma), \quad \inf_{g \in \G} \CR(g,\eps, \gamma),
$$
$$
\CE(\eps, \gamma) = \sup_{g \in \G} \CE(g,\eps, \gamma), \quad \CR(g,\eps, \gamma) = \sup_{g \in \G} \CR(\eps, \gamma)
$$
and
$$
 \essC(\eps,\gamma)=\CE(g_*,\eps, \gamma),\quad\rozC(\eps,\gamma)=\CR(g_*,\eps, \gamma)
$$
for all $\eps,\gamma>0.$ Recall that
$$
g_0(z) = \min\{z, B_n\}, \quad g_1(z) = \max\{z, B_n\},\quad g_*(z)=z,\quad g_{\textsc{c}}(z)=1,\quad z\ge0.
$$

By virtue of invariance of the fractions $\es(g,\eps,\gamma)$,  $\roz(g,\eps,\gamma)$ with respect to scale transform of $g$ and extreme property of $g_1$ (see~\eqref{gmin_gmax_ineq}) we have
$$
\inf_{g\in\G}\CE(g,\eps, \gamma)
=\inf_{g\in\G}\sup_{\substack{F_1,\ldots,F_n\in\mathcal{F}, \\n\in\N\colon B_n>0}} \frac{\Delta_n(F_1,\ldots,F_n)}{\es(g,\eps, \gamma)} 
=\inf_{g\in\G}\sup_{\substack{F_1,\ldots,F_n\in\mathcal{F}, \\n\in\N\colon B_n=1}} \frac{\Delta_n(F_1,\ldots,F_n)}{\es(g,\eps, \gamma)} 
$$
$$
\ge\sup_{\substack{F_1,\ldots,F_n\in\mathcal{F}, \\n\in\N\colon B_n=1}} \frac{\Delta_n(F_1,\ldots,F_n)}{\es(g_1,\eps, \gamma)} 
= \sup_{\substack{F_1,\ldots,F_n\in\mathcal{F}, \\n\in\N\colon B_n>0}} \frac{\Delta_n(F_1,\ldots,F_n)}{\es(g_1,\eps, \gamma)}  = \CE(g_1,\eps, \gamma)
$$
on one hand, and $\inf\limits_{g\in\G}\CE(g,\eps, \gamma)\le\CE(g_1,\eps, \gamma)$  by definition of the lower bound, on the other hand. Therefore
$$
\inf_{g\in\G}\CE(g,\eps, \gamma)=\CE(g_1,\eps, \gamma),\quad \inf_{g\in\G}\CR(g,\eps,\gamma)=\CR(g_1,\eps, \gamma),
$$
where the second equality is proved similar to the first one. 

Let us show that
$$
\CE(\eps, \gamma)= \CE(g_0,\eps, \gamma), \quad \CR(\eps, \gamma) = \CR(g_0,\eps, \gamma).
$$
Indeed, for arbitrary $n \in \N$, $F_1, \ldots, F_n\in\mathcal F$ and $g \in \G$ due to the extremality of~$g_0$ (see~\eqref{gmin_gmax_ineq}) we have
$$
\Delta_n \le \CE(g_0,\eps, \gamma)\es(g_0,\eps, \gamma)\le  \CE(g_0,\eps,\gamma)\es(g,\eps, \gamma),
$$
%$$
%\Delta_n \le \CR(g_0,\eps, \gamma)\roz(g_0,\eps, \gamma)\le  \CR(g_0,\eps, \gamma)\roz(g,\eps, \gamma),
%$$
hence, $\CE(\eps,\gamma)\le\CE(g_0,\eps,\gamma)$, on one hand. On the other hand, this inequality may hold true only with the equality sign, since $\CE(\eps,\gamma)=\sup\limits_{g\in\G}\CE(g,\eps,\gamma)$ by definition. The same reasoning holds also true for~$\CR(\eps, \gamma)$.

Thus, $\CE(g_1,\eps,\gamma)$, $\CR(g_1,\eps,\gamma)$ are the most optimistic constants, while $\CE(g_0,\eps,\gamma)$, $\CR(g_0,\eps,\gamma)$ are the most pessimistic, but universal (exact) ones. The next theorem establishes lower bounds for the exact constants $\CE(\eps,\gamma)=\CE(g_0, \eps, \gamma)$, $\CE(\eps,\gamma)=\CR(g_0, \eps, \gamma)$, and also for the constants $\essC(\eps,\gamma)$ and $\rozC(\eps,\gamma)$ appearing in inequalities~$\eqref{EsseenTypeIneq2018}$ and  $\eqref{RozovskiiTypeIneq2018}$.

\begin{Theorem} \label{theorem_sup_abs_C}
{\rm (i)}
For all $\eps,\gamma>0$ we have
$$
%\min\{\CE(g_0,\eps,\gamma), \CR(g_0,\eps,\gamma)\}
\min\{
\CE(\eps,\gamma), \CR(\eps,\gamma), \essC(\eps,\gamma),\rozC(\eps,\gamma)\} %\ge\frac{\Delta_1(\tfrac12)}{\min\{1,\eps\}}
\ge\frac{\Phi(1)-0.5}{\min\{1,\eps\}} 
>\frac{0.3413}{\min\{1,\eps\}}.
$$
{\rm (ii)} For $\frac{1}{2} < p < 1,$ $q=1-p$ set
$$
\Delta_1(p):=\Phi(\sqrt{q/p})-q.
$$
Take any $\gamma>0$. Let us denote for $\eps\le1:$
$$
K_{\scE,0}(p, \eps, \gamma) = 
\begin{cases}
\Delta_1(p) / \eps, & \frac{1}{2} < p \le \frac{1}{\eps^2 + 1}, 
\\
\Delta_1(p) \big/ \max\big\{\sqrt{q/p}, \gamma \sqrt{q^3/p} + \eps p\big\}, & \frac{1}{\eps^2 + 1} < p < 1,
\end{cases}
$$
$$
K_{\scR,0}(p, \eps, \gamma) = 
\begin{cases}
\Delta_1(p) / \eps, & \frac{1}{2} < p \le \frac{1}{\eps^2 + 1}, 
\\
\Delta_1(p) \big/ \big( \gamma \sqrt{q^3/p} + \max\big\{ \sqrt{q/p}, \eps p \big\} \big), & \frac{1}{\eps^2 + 1} < p < 1,
\end{cases}
$$
and for $\eps>1:$
$$
K_{\scE,0}(p, \eps, \gamma) = 
\Delta_1(p)/ 
\max\Big\{ \sqrt{{q}/{p}}, \gamma \sqrt{{q^3}/{p}} + p \Big\}, \quad \frac12<p<1, 
$$
$$
K_{\scE,*}(p, \eps, \gamma) = 
\begin{cases}
\Delta_1(p) / \max\Big\{ \sqrt{\frac{q}{p}}, \gamma \sqrt{\frac{q^3}{p}} + \eps p \Big\}, & \frac{\eps^2}{\eps^2+1}\le p<1, 
\\
\Delta_1(p)\sqrt{pq}\big/ \max\left\{q, \gamma q^2 + p^2,\gamma(p-q)\right\}, & \frac12<p<\frac{\eps^2}{\eps^2+1},
\end{cases}
$$
$$
K_{\scR,0}(p, \eps, \gamma) = 
\begin{cases}
\Delta_1(p) \Big/ \Big(\frac{\gamma}{\eps}\sqrt{\frac{q^3}{p}}
+\max\left\{ \sqrt{\frac{q}{p}}, p \right\}\Big), & \frac{\eps^2}{\eps^2+1}\le p<1, 
\\
\Delta_1(p) \Big/ \Big(\frac{\gamma(p-q)}{\eps\sqrt{pq}}
+\max\left\{ \sqrt{\frac{q}{p}}, p \right\}\Big), & \frac12<p<\frac{\eps^2}{\eps^2+1},
\end{cases}
$$
$$
K_{\scR,*}(p, \eps, \gamma) = 
\begin{cases}
\Delta_1(p) \Big/ \Big(\gamma\sqrt{\frac{q^3}{p}}+\max\left\{ \sqrt{\frac{q}{p}}, \eps p \right\} \Big), & \frac{\eps^2}{\eps^2+1}\le p<1, 
\\
\Delta_1(p) \Big/ \Big(\frac{\gamma(p-q)}{\sqrt{pq}}+ \max\Big\{ \sqrt{\frac{q}{p}}, \sqrt{\frac{p^3}{q}} \Big\}\Big), & \frac12<p<\frac{\eps^2}{\eps^2+1}.
\end{cases}
$$
Then for all $\eps,\gamma>0$ we have
$$
\CE(\eps, \gamma) \ge \sup_{\frac12 < p < 1} K_{\scE,0}(p, \eps, \gamma),\quad 
\CR(\eps, \gamma) \ge \sup_{\frac12 < p < 1} K_{\scR,0}(p, \eps, \gamma),
$$
and for $\eps>1$ also
%$$
%\CE(\eps, \gamma)=\essC(\eps,\gamma) \ge \sup_{\frac12 < p < 1} K_{\scE,0}(p, \eps, \gamma),
%$$
%$$
%\CR(\eps, \gamma) =\rozC(\eps,\gamma) \ge \sup_{\frac12 < p < 1} K_{\scR,0}(p, \eps, \gamma),
%$$
$$
\essC(\eps,\gamma)\ge\sup_{\frac12<p<1} K_{\scE,*}(p, \eps, \gamma),\quad \rozC(\eps,\gamma)\ge\sup_{\frac12<p<1} K_{\scR,*}(p, \eps, \gamma).
$$
In particular, 
$$
\essC(\infty,1)\ge\sup_{\frac12<p<1} \frac{\big(\Phi(\sqrt{q/p}) - q\big)\sqrt{pq} }{q^2 + p^2}> 0.3703,
$$
and for the both constants $C_{\bullet}\in\{\CE,\CR\}$ with $\eps = 1, \gamma = 1$ the following lower bound holds:
$$
%\min\{\CE(1,1),\CR(1,1),\CE(\infty,1)\}\ge
%C_{\bullet}(1,1)\ge
\min\{C_{\bullet}(1,1),\CE(\infty,1)\}\ge
\sup_{\frac12<p<1} \frac{\Phi(\sqrt{q/p}) - q}{\sqrt{{q^3}/{p}}+p} >0.5685\quad (p = 0.9058\ldots).
$$
\end{Theorem}

\begin{proof}
Let $n=1$ and the r.v. $X_1$ have the two-point distribution~\eqref{two_point_dist} with ${p\in[\frac12,1)}$. Then, by theorem~\ref{ThUniDist(2point,Phi)n=1}, we have $\Delta_1(F_1)=\Delta_1(p)$, and all the constants $C_{\bullet}\in\{\CE,\CE\}$, $A_{\bullet}\in\{\essC,\rozC\}$ can be bounded from below as
$$
C_{\bullet}(\eps,\gamma)\ge \sup_{\frac12\le p<1} \frac{\Delta_1(p)}{L_{\bullet,1}(g_0,\eps,\gamma)},\quad 
A_{\bullet}(\eps,\gamma)\ge \sup_{\frac12\le p<1} \frac{\Delta_1(p)}{L_{\bullet,1}(g_*,\eps,\gamma)}.
$$
In particular, for $p=\frac12$ we have $\Delta_1(\tfrac12)=\Phi(1)-0.5=0.3413\ldots,$ and by theorem~\ref{ThFractionsFor2PointDistr}
$$
\es(g_0,\eps, \gamma)= \roz(g_0,\eps, \gamma)= \es(g_*,\eps, \gamma)= \roz(g_*,\eps, \gamma)=\min\{1,\eps\},\quad \eps,\gamma>0,
$$
whence the statement of point (i) follows immediately. 

To prove point (ii) it suffices to make sure that for all $p\in(\frac12,1)$ and $\eps,\gamma>0$
$$
K_{\bullet,0}(p,\eps,\gamma)= \frac{\Delta_1(p)}{L_{\bullet,1}(g_0,\eps,\gamma)}, \quad 
K_{\bullet,*}(p,\eps,\gamma)= \frac{\Delta_1(p)}{L_{\bullet,1}(g_*,\eps,\gamma)},\quad \bullet\in\{\scE,\scR\}.
$$
Theorem~\ref{ThFractionsFor2PointDistr} (i) with $n=1$ implies that for $p>\frac12$, $\eps,\gamma>0$
$$
L_{\scE,1}(g_*,\eps,\gamma)
=\left\{\begin{array}{ll}
\eps, & \eps^2\le\frac{q}{p},
%&\Leftrightarrow& \frac12<p\le\frac1{\eps^2+1}, 
\\
\max\left\{ \sqrt{\frac{q}{p}}, \gamma \sqrt{\frac{q^3}{p}} + \eps p \right\}, & {\frac{q}{p}}<\eps^2\le\frac{p}{q},
%&\Leftrightarrow&  \frac{\eps^2\vee1}{\eps^2+1}\le p<1, 
\\
\max\left\{q, \gamma q^2 + p^2,\gamma(p-q)\right\}/\sqrt{pq}, & \eps^2>\frac{p}{q}.
%&\Leftrightarrow& \frac12<p\le\frac{\eps^2}{\eps^2+1}.
\end{array}\right.
$$
$$
L_{\scR,1}(g_*,\eps,\gamma)
=\begin{cases}
\eps, & \eps^2\le\frac{q}{p}, 
\\
\gamma\sqrt{\frac{q^3}{p}}+\max\left\{ \sqrt{\frac{q}{p}}, \eps p \right\}, & \frac{q}{p}<\eps^2\le\frac{p}{q}, 
\\
\frac{\gamma(p-q)}{\sqrt{pq}}+ \max\Big\{ \sqrt{\frac{q}{p}}, \sqrt{\frac{p^3}{q}} \Big\}, & \eps^2>\frac{p}{q},
\end{cases}
$$
$L_{\scE,1}(g_0,\eps,\gamma)=L_{\scE,1}(g_*,\eps\wedge1,\gamma)$, while for $\eps>1$, by the same theorem~\ref{ThFractionsFor2PointDistr} (iii), we have
$$
L_{\scR,1}(g_0,\eps,\gamma)
=\begin{cases}
\frac{\gamma}{\eps}\sqrt{\frac{q^3}{p}}
+\max\left\{ \sqrt{\frac{q}{p}}, p \right\}
,& 1<\eps^2\le\frac{p}{q}, 
\\
\frac{\gamma(p-q)}{\eps\sqrt{pq}}
+\max\left\{ \sqrt{\frac{q}{p}}, p \right\}
,& 1\le\frac{p}{q}<\eps^2, 
\end{cases}
$$
(recall  that $\roz(g_0,\eps,\gamma)=\roz(g_*,\eps,\gamma)$ for $\eps\le1$). Note that for all $\eps>0$ and $p\in(\frac12,1)$
\begin{equation}\label{eps^,q,p:3inequalities}
\begin{array}{lclcl}
\eps^2\le\frac{q}{p}
&\Leftrightarrow& 
p\in\big(\frac12,\frac1{\eps^2+1}\big]
&=&\emptyset,\ \eps>1, 
\\
{\frac{q}{p}}<\eps^2\le\frac{p}{q}
&\Leftrightarrow&  
\begin{cases}
p>\frac{1}{\eps^2+1},
\\
p\ge\frac{\eps^2}{\eps^2+1},
\end{cases}
&\Leftrightarrow& 
\begin{cases}
p\ge\frac{\eps^2}{\eps^2+1},\quad \eps>1,
\\
p>\frac{1}{\eps^2+1},\quad\eps\le1,
\end{cases}
\\
\eps^2>\frac{p}{q}
&\Leftrightarrow&
p\in\big(\frac12,\frac{\eps^2}{\eps^2+1}\big)
&=&\emptyset,\  \eps\le1,
\end{array}
\end{equation}
so that from the three conditions in~\eqref{eps^,q,p:3inequalities} under the additional condition $p\in(\frac12,1)$ there remain only two:
$$
%\left[
\begin{array}{l}
\frac12<p\le\frac1{\eps^2+1}
\\
\frac{1}{\eps^2+1}<p<1
\end{array}
%\right.
\quad\text{for }\eps\le1,
\qquad 
%\left[
\begin{array}{l}
\frac{\eps^2}{\eps^2+1}\le p<1
\\
\frac12<p<\frac{\eps^2}{\eps^2+1}
\end{array}
%\right.
\quad\text{for }\eps>1.\qquad 
$$
In particular, for $\eps\le1$
$$
L_{\scE,1}(g_0,\eps,\gamma)=L_{\scE,1}(g_*,\eps,\gamma)
=\begin{cases}
\eps,& \frac12<p\le\frac1{\eps^2+1}, 
\\
\max\Big\{ \sqrt{\frac{q}{p}}, \gamma \sqrt{\frac{q^3}{p}} + \eps p \Big\}, &  \frac{1}{\eps^2+1}<p<1, 
\end{cases}
$$
$$
L_{\scR,1}(g_0,\eps,\gamma)=L_{\scR,1}(g_*,\eps,\gamma)
=\begin{cases}
\eps, &  \frac12<p\le\frac1{\eps^2+1},
\\
\gamma\sqrt{\frac{q^3}{p}}+\max\left\{ \sqrt{\frac{q}{p}}, \eps p \right\}, & \frac{1}{\eps^2+1}<p<1, 
\end{cases}
$$
and for $\eps>1$
$$
L_{\scE,1}(g_*,\eps,\gamma)
=\left\{\begin{array}{ll}
\max\Big\{ \sqrt{\frac{q}{p}}, \gamma \sqrt{\frac{q^3}{p}} + \eps p \Big\}, & \frac{\eps^2}{\eps^2+1}\le p<1, 
\\
\max\left\{q, \gamma q^2 + p^2,\gamma(p-q)\right\}/\sqrt{pq}, &  \frac12<p\le\frac{\eps^2}{\eps^2+1}.
\end{array}\right.
$$
$$
L_{\scE,1}(g_0,\eps,\gamma)=L_{\scE,1}(g_*,1,\gamma)=
\max\Big\{ \sqrt{\tfrac{q}{p}}, \gamma \sqrt{\tfrac{q^3}{p}} + p \Big\},\quad \tfrac12<p<1,
$$
$$
L_{\scR,1}(g_*,\eps,\gamma)
=\begin{cases}
\gamma\sqrt{\frac{q^3}{p}}+\max\left\{ \sqrt{\frac{q}{p}}, \eps p \right\}, &  \frac{\eps^2}{\eps^2+1}\le p<1, 
\\
\frac{\gamma(p-q)}{\sqrt{pq}}+ \max\Big\{ \sqrt{\frac{q}{p}}, \sqrt{\frac{p^3}{q}} \Big\}, &  \frac12<p<\frac{\eps^2}{\eps^2+1},
\end{cases}
$$
$$
L_{\scR,1}(g_0,\eps,\gamma)
=\begin{cases}
\frac{\gamma}{\eps}\sqrt{\frac{q^3}{p}}
+\max\left\{ \sqrt{\frac{q}{p}}, p \right\}
,& \frac{\eps^2}{\eps^2+1}\le p<1, 
\\
\frac{\gamma(p-q)}{\eps\sqrt{pq}}
+\max\left\{ \sqrt{\frac{q}{p}}, p \right\}
,& \frac12<p<\frac{\eps^2}{\eps^2+1},
\end{cases}
$$
whence we obtain the above expressions for $K_{\bullet,0}$, $K_{\bullet,*}$.

Now let us consider the particular cases. The coinciding lower bounds for $\CE(\infty,1)$, $\CE(1,1)$ and $\CR(1,1)$ follow from that  $L_{\scE,1}(g_0,\infty,1)=\sqrt{{q^3}/{p}} + p$ for all $p\in(\frac12,1)$, and also from, as it can easily be made sure,
$$
L_{\scE,1}(g_0,1,1)=L_{\scE,1}(g_0,\infty,1)=\sqrt{q^3/p}+p=L_{\scR,1}(g_0,1,1)
$$
for $p\ge p_0$, where $p_0 = 0.6823\ldots$ is the unique root of the equation  $p^3 + p - 1 = 0$ on the segment $[0.5,1]$.  The lower bound for $\essC(\infty,1)$ follows from that $L_{\scE,1}(g_*,\infty,1)=({q^2+p^2})/{\sqrt{pq}}$.
\end{proof}

Now let us find lower bouds for the most optimistic constants $\CE(g_1,\eps, \gamma)$ and~$\CR(g_1,\eps, \gamma)$.

\begin{Theorem} \label{theorem_inf_abs_C}
{\rm (i)} For all $\eps,\gamma>0$ we have
$$
\min\{\CE(g_1, \eps, \gamma), \CR(g_1, \eps, \gamma)\}\ge
\Phi(1)-0.5
%\Delta_1(\tfrac12)
>0.3413.
$$
{\rm (ii)} For $\frac{1}{2} < p < 1,$ $q=1-p$ set
$$
\Delta_1(p):=\Phi(\sqrt{q/p})-q.
$$
Take any $\gamma>0$. Let us denote for $\eps\le1:$
$$
K_{\scE,1}(p, \eps, \gamma) = 
\begin{cases}
\Delta_1(p), &\frac{1}{2} < p \le \frac{1}{\eps^2 + 1}, 
\\
\Delta_1(p) / \max\{1, \gamma q + p\}, &\frac{1}{\eps^2 + 1} < p < 1,
\end{cases}
$$
$$
K_{\scR,1}(p, \eps, \gamma) = \begin{cases}
\Delta_1(p), &\frac{1}{2} < p \le \frac{1}{\eps^2 + 1}, 
\\
\Delta_1(p) / \Big( 1 + \frac{\gamma}{\eps}\sqrt{\frac{q^3}{p}}\,\Big), & \frac{1}{\eps^2 + 1} < p < 1,
\end{cases}
$$
and for $\eps>1:$
$$
K_{\scE,1}(p, \eps, \gamma) = 
\begin{cases}
\Delta_1(p)/\max\Big\{1, \gamma q + p, \frac{\gamma q^2+p^2}{\sqrt{pq}}, \frac{\gamma(p - q)}{\sqrt{pq}}\Big\}, &\frac{1}{2} <p<\frac{\eps^2}{\eps^2+1}, 
\\
\Delta_1(p) / \max\big\{1, \gamma q + p, \frac{\gamma q^{3/2}}{\sqrt{p}}+p\eps \big\}, &\frac{\eps^2}{\eps^2+1}\le p < 1,
\end{cases}
$$
$$
K_{\scR,1}(p, \eps, \gamma) = \begin{cases}
\Delta_1(p)/\big(\frac{\gamma(p-q)}{\sqrt{pq}}+\max\big\{ \sqrt{\frac{p^3}{q}},1\big\}\big),  &\frac{1}{2} <p<\frac{\eps^2}{\eps^2+1}, 
\\
\Delta_1(p) /\big(\gamma\sqrt{\frac{q^3}{p}}+\max\left\{ \eps p,1\right\}\big),  &\frac{\eps^2}{\eps^2+1}\le p < 1.
\end{cases}
$$
Then for all $\eps,\gamma>0$ we have
$$
\CE(g_1,\eps, \gamma) \ge \sup_{\frac12 < p < 1}  K_{\scE,1}(p, \eps, \gamma), \quad \CR(g_1,\eps, \gamma) \ge \sup_{\frac12 < p < 1} K_{\scR,1}(p, \eps, \gamma).
$$
In particular, 
$$
\CE(g_1,\eps, \gamma) \ge \sup_{\frac12 < p < 1} \Delta_1(p)=\Delta_1(p_{\Phi}) =: C_{\Phi}=0.5409\ldots
$$
for $\gamma \le 1$ or $\eps \le x_{\Phi},$ where $x_{\Phi}= 0.213105\ldots$ is defined in lemma~$\ref{ChebotarevPsi}$,  $p_{\Phi} = (x_{\Phi}^2 + 1)^{-1} = 0.9565\ldots,$
$$
\CE(g_1,\infty,1)\ge\sup_{\frac12<p<1} \frac{\big(\Phi(\sqrt{q/p}) - q\big)\sqrt{pq} }{q^2 + p^2}> 0.3703,
$$
$$
\CR(g_1,\eps, \gamma) \ge C_{\Phi}
%\sup_{\frac12 < p < 1} \Delta_1(p) = C_{\Phi} = 0.54093\ldots,
\quad \text{for}\quad \eps \le x_{\Phi},
$$
$$
\CR(g_1,1, 1) \ge \sup_{\frac12 < p < 1} K_{\scR,1}(p, 1, 1) = \frac{\Delta_1(p)}{1 + \sqrt{(1-p)^3/p}}\bigg|_{p = 0.9678\ldots}> 0.5370.
$$
\end{Theorem}

\begin{proof} 
Let $n=1$ and the r.v. $X_1$ have the two-point distribution~\eqref{two_point_dist} with ${p\in[\frac12,1)}$. Then, by theorem~\ref{ThUniDist(2point,Phi)n=1} we have $\Delta_1(F_1)=\Delta_1(p)$, and all the constants $C_{\bullet}\in\{\CE,\CE\}$ are bounded from below as
$$
C_{\bullet}(g_1,\eps,\gamma)\ge \sup_{\frac12\le p<1} \frac{\Delta_1(p)}{L_{\bullet,1}(g_1,\eps,\gamma)}.
$$
In particular, for $p=\frac12$ we have $\Delta_1(\tfrac12)=\Phi(1)-0.5=0.3413\ldots,$ while by theorem~\ref{ThFractionsFor2PointDistr}
$$
\es(g_1,\eps, \gamma)= \roz(g_1,\eps, \gamma)\equiv1,\quad \eps,\gamma>0,
$$
whence the statement of point (i) follows immediately. 

To prove point  (ii), it suffices to make sure that for all $p\in(\frac12,1)$, $\eps,\gamma>0$ we have
$$
K_{\bullet,1}(p,\eps,\gamma)= \frac{\Delta_1(p)}{L_{\bullet,1}(g_1,\eps,\gamma)}, \quad 
\bullet\in\{\scE,\scR\}.
$$
Theorem~\ref{ThFractionsFor2PointDistr} (i) (see~\eqref{Esseen_g1_two_point} and~\eqref{Rozovsky_g1_two_point}) with $n=1$ implies that for $p>\frac12$, $\gamma>0$ and $\eps\le1$ we have
$$
L_{\scE,1}(g_1,\eps,\gamma)=L_{\scE,1}(g_{\textsc{c}},\eps,\gamma)
= \begin{cases}
1, & \eps^2\le\frac{q}{p}\quad\Leftrightarrow\quad p\le\frac{1}{\eps^2+1},
\\
\max\left\{ 1, \gamma q + p \right\}, & \eps^2>\frac{q}{p}\quad\Leftrightarrow\quad p>\frac{1}{\eps^2+1},
\end{cases}
\ =\ \frac{\Delta_1(p)}{K_{\scE,1}(p, \eps, \gamma)},
$$
$$
L_{\scR,1}(g_1,\eps,\gamma)=L_{\scR,1}(g_{\textsc{c}},\eps,\gamma) =
\begin{cases}
1, & \eps^2\le\frac{q}{p},
%\quad\Leftrightarrow\quad p\le\frac{1}{\eps^2+1},
\\
\frac{\gamma}{\eps}\sqrt{\frac{q^3}{p}}+1, & \eps^2>\frac{q}{p}, 
\end{cases}
\ =\ \frac{\Delta_1(p)}{K_{\scR,1}(p, \eps, \gamma)},
$$
while for $\eps>1$, by theorem~\ref{ThFractionsFor2PointDistr} (ii) and (iii), respectively, we have
$$
L_{\scE,1}(g_1,\eps,\gamma) = 
\begin{cases}
\max\big\{1, \gamma q + p, \frac{\gamma q^{3/2}}{\sqrt{p}}+p\eps \big\}, & \eps^2\le \frac pq\quad\Leftrightarrow\quad p\ge\frac{\eps^2}{\eps^2+1},
\\
\max\big\{1, \gamma q + p, \frac{\gamma q^2+p^2}{\sqrt{pq}}, \frac{\gamma(p - q)}{\sqrt{pq}}\big\}, & \eps^2>\frac pq\quad\Leftrightarrow\quad p<\frac{\eps^2}{\eps^2+1},
\end{cases}
\ =\ \frac{\Delta_1(p)}{K_{\scE,1}(p, \eps, \gamma)},
$$
$$
L_{\scR,1}(g_1,\eps,\gamma) =
\begin{cases}
\gamma\sqrt{\frac{q^3}{p}}+\max\left\{ \eps p,1\right\},
&\eps^2\le\frac{p}{q},
%\quad\Leftrightarrow\quad p\ge\frac{\eps^2}{\eps^2+1}, 
\\
\frac{\gamma(p-q)}{\sqrt{pq}}+\max\left\{ \sqrt{\frac{p^3}{q}},1\right\},
& \eps^2>\frac{p}{q},
%\quad\Leftrightarrow\quad p<\frac{\eps^2}{\eps^2+1}.
\end{cases}
\ =\ \frac{\Delta_1(p)}{K_{\scR,1}(p, \eps, \gamma)},
$$
which coincides with the statement of point (ii).

Now let us consider the particular cases. The lower bound for $\CE(g_1,\eps,\gamma)$ with $\gamma\le1$ follows from that, for the specified $\gamma$ and $\eps\le1$, we have $K_{\scE,1}(p, \eps, \gamma) = \Delta_1(p)$, $\frac12<p\le1$. For $\eps\le x_{\Phi}$ we have $\frac1{\eps^2+1}\ge\frac1{x_{\Phi}^2+1}=p_{\Phi}$ and, hence,
$$
\CE(g_1,\eps, \gamma)\ge 
%\sup_{\frac12 < p < 1} K_{\scE,1}(p, \eps, \gamma)\ge 
\sup_{\frac12 < p\le\frac1{\eps^2+1}} K_{\scE,1}(p, \eps, \gamma) = \sup_{\frac12 < p\le\frac1{\eps^2+1}}\Delta_1(p)=\Delta_1(p_{\Phi}) = C_{\Phi}.
$$
The lower bound for $\CR(g_1,\eps, \gamma)\ge C_{\Phi}$ with $\eps\le x_{\Phi}$ is obtained similarly. The given values of the constants $\CE(g_1,\infty,1)$, $\CR(g_1,1, 1) $ are computed trivially.
\end{proof}

\section{Lower bounds for he asymptotically best constants}

Let us investigate the constants $\abe(g, \eps, \gamma)$ in~\eqref{AW_ineq} and~\eqref{Rozovsky_ineq} and construct their lower bounds.
Due to the extremality of the functions $g_0(z)=\min\{z,B_n\}$ and $g_1(z)=\max\{z,B_n\}$ (see~\eqref{gmin_gmax_ineq}) we have
$$
\sup_{g \in \G} \abe(g, \eps, \gamma) = \abe(g_0, \eps, \gamma), \quad \inf_{g \in \G} \abe(g, \eps, \gamma) = \abe(g_1, \eps, \gamma).
$$
in both inequalities~\eqref{AW_ineq} and \eqref{Rozovsky_ineq}.

\begin{Theorem} \label{theorem_AB_lb}
For all $\eps,\ \gamma > 0:$
\\
{\rm (i)} in inequality~\eqref{AW_ineq}
\begin{equation}\label{ABEEsseenLowBound}
\inf_{\eps>0}\abe(g_0, \eps, \gamma) \ge \frac1{3\sqrt{2\pi}}\sup_{\frac12 < p < 1} \frac{p+1}{\max\{ 1-p, \gamma (1-p)^2 + p^2, \gamma(2p - 1) \}}\eqqcolon \underline{\abe}(\gamma) ,
\end{equation}
in particular, $\inf\limits_{\eps>0}\abe(g_0, \eps, 1) \ge \frac{\sqrt{10} + 3}{6\sqrt{2\pi}} = 0.4097\ldots;$
\\
{\rm (ii)} in inequality \eqref{Rozovsky_ineq}
\begin{equation}\label{ABERozLowBound}
\abe(g_0, \eps, \gamma) \ge
\begin{cases}
\displaystyle
\frac{1 + \sqrt{5}}{3\sqrt{2\pi} \big(2\gamma(\eps^{-1}\wedge1)(\sqrt{5} - 2) + 3 - \sqrt{5}\big)}, &\displaystyle
\frac\gamma{\eps\vee1} <\frac{2}{3}, 
\\[3mm]
\displaystyle\frac{1}{\sqrt{2\pi}}=0.3989\ldots,
&\displaystyle
\frac\gamma{\eps\vee1} \ge \frac{2}{3};
\end{cases}
\end{equation}
\\
{\rm (iii)} in both inequalities \eqref{AW_ineq} and~\eqref{Rozovsky_ineq}
$$
\abe(g_1, \eps, \gamma) = 0.
$$
\end{Theorem}

Values of $\underline{\abe}(\gamma)$ for some $\gamma$ and the corresponding extreme values of~$p$ are given in table~\ref{Tab:EsseenABELowBound(gamma)}.

\begin{table}[h]
\begin{center}
\begin{tabular}{||c||c|c|c|c|c|c|c|c|c|c||}
\hline
$\gamma$&$0.1$&$0.2$&$0.4$&$0.56$&$1$&$1.5$&$2$&$3$&$4$&$5$
\\\hline
$p$&$0.6112$&$0.6039$&$0.5871$&$0.5710$ &$0.5812$&$0.6733$&$0.6666$&$0.6340$&$0.6202$&$0.6126$
\\\hline
$\underline{\abe}(\gamma)$&$0.5511$&$0.5384$&$0.5111$&$0.4868$ &$0.4097$&$0.3627$&$0.3324$&$0.2703$&$0.2240$&$0.1904$
\\\hline
\end{tabular}
\end{center}
\caption{Values of the lower bound $\underline{\abe}(\gamma)$ (see~\eqref{ABEEsseenLowBound}), rounded down, for the asymptotically best constant $\abe(g_0,\eps,\gamma) $ from inequality~\eqref{AW_ineq} for some~$\gamma$. The second line contains rounded extreme values of~$p$ in~\eqref{ABEEsseenLowBound}.}
\label{Tab:EsseenABELowBound(gamma)}
\end{table}

\begin{proof}
Let us show that $\abe(g_1, \eps, \gamma) = 0$. 
According to~\eqref{g_1_Esseen} and~\eqref{g_1_Rozovsky}, ${\es(g_1,\eps, \gamma) \ge 1}$, $\roz(g_1,\eps, \gamma) \ge 1$ for all $\eps,\gamma>0$, so that for the both inequalities \eqref{AW_ineq} and~\eqref{Rozovsky_ineq} we have $\abe(g_1, \eps, \gamma) \le \sup\limits_{F} \limsup\limits_{n \to \infty} \Delta_n(F) = 0,$ whence, with the account of non-negativity of $\abe(g_1, \eps, \gamma)$, the statement of point (iii) follows.

Now let us estimate from below the constants $\abe(g_0, \eps, \gamma)$. Take i.i.d. r.v.'s  $X_1,\ldots,X_n$ with distribution~\eqref{two_point_dist}, where $p>1/2$. Then, by virtue of theorem~\ref{ThFractionsFor2PointDistr}, for $n(\eps^2\wedge1) > \frac{p}{q}$ we have
$$
\es(g_0,\eps, \gamma) = 
%\max\left\{\sqrt{\frac{q}{np}}, \frac{\gamma q^2 + p^2}{\sqrt{npq}}, \frac{\gamma(p - q)}{\sqrt{npq}}  \right\} = 
\frac{\max\{ q, \gamma q^2 + p^2, \gamma(p-q) \}}{\sqrt{npq}}.
$$

The r.v. $X_1$ is lattice with a span $h = {1}/{\sqrt{pq}}$. The Esseen asymptotic expansion~\cite{Esseen1956} for lattice distributions with the span $h$ implies that
$$
\limsup_{n \rightarrow \infty} \Delta_n \sqrt{n} = \frac{\abs{\E X_1^3} + 3h\sigma_1^2}{6\sqrt{2\pi}\sigma_1^3}.
$$
For the chosen distribution of $X_1$ we have $\sigma_1^2=1$, $\E X_1^3=(p-q)/\sqrt{pq}$, therefore
$$
\lim_{n \rightarrow +\infty}\Delta_n \sqrt{n} = \frac{p-q+3}{6\sqrt{2\pi pq}} = \frac{p+1}{3\sqrt{2\pi pq}},
$$
and, hence, in inequality \eqref{AW_ineq} we have
$$
\abe(g_0, \eps, \gamma)\ge \lim_{n\to\infty} \frac{\Delta_n}{\es(g_0,\eps, \gamma)} = \lim_{n\to\infty} \frac{\Delta_n \sqrt{npq}}{\max\{ q, \gamma q^2 + p^2, \gamma(p-q) \}}  =
$$
$$= \frac{p+1}{3\sqrt{2\pi} \cdot \max\{ q, \gamma q^2 + p^2, \gamma(p-q) \}}
$$
for all  $p\in(1/2,1)$, $\eps,\gamma>0$, whence, with the account of arbitrariness of the choice of~$p$, the statement of point (ii) follows. In particular, for $\gamma = 1$ we obtain
$$
\abe(g_0, \eps,1)\ge \sup_{\frac12 < p < 1} \frac{p+1}{3\sqrt{2\pi} (p^2 + q ^ 2)} 
=\frac1{6\sqrt{2\pi}}\sup_{\frac12 < p < 1} \frac{\abs{\E X_1^3}+3h\sigma_1^2}{\E|X_1|^3} 
=\frac{\sqrt{10} + 3}{6\sqrt{2\pi}} = 0.4097\ldots
$$
for $p = \sqrt{5/2} - 1 = 0.5811\ldots$ (extremality of the specified $p$ was proved in~\cite{Esseen1956}).

Now let us consider inequality \eqref{Rozovsky_ineq}. Taking into account that for $n(\eps^2\wedge1)> \frac{p}{q}$ we have
$$
\roz(g_0,\eps, \gamma) = \frac{\gamma(p - q)(\eps^{-1}\wedge1) + \max\{q, p^2\}}{\sqrt{npq}},
\quad
\max\{q, p^2\}=
\begin{cases}
q,& p \in (\frac{1}{2},\frac{\sqrt{5} - 1}{2}],
\\
p^2& p \in (\frac{\sqrt{5} - 1}{2}, 1),
\end{cases}
$$
and denoting
$$
a=\gamma(\eps^{-1}\wedge1)=\frac{\gamma}{\eps\vee1},
$$
we obtain
$$
f(p)\coloneqq 3\sqrt{2\pi} \lim_{n\to\infty} \frac{\Delta_n}{\roz(g_0,\eps, \gamma)} 
%= \frac{p + 1}{3\sqrt{2\pi pq}} \cdot \frac{\sqrt{pq}}{\gamma(p - q)(\eps^{-1}\wedge1)  + \max\{q, p^2\}} =
%$$ 
%$$ 
%=\frac{p + 1}{3\sqrt{2\pi}(\gamma(p - q)(\eps^{-1}\wedge1)  + \max\{q, p^2\})} 
= \begin{cases}
\displaystyle \frac{p + 1}{a(2p - 1)  + 1 - p}, 
&\displaystyle
p \in \bigg(\frac{1}{2},\frac{\sqrt{5} - 1}{2}\bigg], 
\\[3mm]
\displaystyle \frac{p + 1}{a(2p - 1)  + p^2}, 
&\displaystyle
p \in \bigg(\frac{\sqrt{5}-1}{2},1\bigg).
\end{cases}
$$
Let us investigate the behaviour of $f(p)$ in dependence of $a>0$. For $p\in\big(\frac{1}{2},\frac{\sqrt{5}-1}{2}\big]$ the numerator of $f'(p)$ takes the form
$$
a(2p - 1) + 1 - p - (p + 1)(2a - 1) = 2 - 3a,
$$
so that $f(p)$ is monotonically decreasing, if $a>\frac{2}{3}$, and monotonically increasing, if $a<\frac{2}{3}$, while for $p \in \big(\frac{\sqrt{5}-1}{2},1\big)$ the numerator of  $f'(p)$ has the form
$$
a(2p-1) + p^2 - 2(p+1)(a + p) = -p^2 - 2p - 3a < 0,\quad 1/2<p<1,
$$
and hence, $f(p)$ decreases strictly monotonically. Therefore,
$$
3\sqrt{2\pi}\cdot\abe(g_0, \eps, \gamma) \ge \sup_{1/2<p<1}f(p)=
\begin{cases}
\displaystyle f\bigg(\frac{\sqrt{5}-1}{2}\bigg)
=\frac{1 + \sqrt{5}}{2a(\sqrt{5} - 2) + 3 - \sqrt{5}},
& a<\frac{2}{3}, 
\\[3mm]
f\big(\frac12+\big)=3,
& a\ge\frac{2}{3},
\end{cases}
$$
whence the statement of point (ii) follows immediately.
\end{proof}

\section{Lower bounds for the asymptotically exact constants}

First of all note that asymptotically exact constants $\aex(g_1, \eps, \gamma)$ and the lower asymptotically exact constants $\lowaex(g_1, \eps, \gamma)$ are defined for none of the inequalities~\eqref{AW_ineq}, \eqref{Rozovsky_ineq}, since the corresponding fractions $\es(g_1,\eps,\gamma)$, $\roz(g_1,\eps,\gamma)$ are  bounded from below by one uniformly with respect to $\eps$ and $\gamma$ (see~\eqref{g_1_Esseen} and~\eqref{g_1_Rozovsky}) and, hence, cannot be infinitesimal.

\begin{Theorem} \label{theorem_lae_lb}
In the both inequalities \eqref{AW_ineq}, \eqref{Rozovsky_ineq} for all $\eps,\ \gamma > 0$  we have
\begin{equation}\label{LowAEX-low-bound}
\lowaex(g_0, \eps, \gamma) \ge \frac{1}{2\sqrt{2\pi}},
\end{equation}
$$
\sup_{g \in G} \upaex(g, \eps, \gamma) = \upaex(g_0, \eps, \gamma) \ge \upaexcond(g_0, \eps, \gamma) \ge \frac{1}{2(\eps\wedge1)},
$$
$$
\inf_{g \in G} \upaex(g, \eps, \gamma) = \upaex(g_1, \eps, \gamma) \ge \upaexcond(g_1, \eps, \gamma) \ge
\begin{cases}
0.5,
&\eps\le1,
\\
\exp\{-\eps^{-2}\}I_0(\eps^{-2})/2,
&1<\eps\le1.1251,
\\
0.2344,&\eps\ge1.1251.
\end{cases}
$$
where $I_0(z)=\sum_{k=0}^{\infty} {(z/2)^{2k}}/{(k!)^2} $ is the modified Bessel function of the zero order.
\end{Theorem}

As a lower bound to the asymptotically exact constant $\aex(g_0,\eps,\gamma)$, due to~\eqref{asypt_const_in}, both the asymptotically best $\abe(g_0,\eps,\gamma)$ and the lower asymptotically exact $\lowaex(g_0,\eps,\gamma)$ constants may serve. But the lower bound of the lower asymptotically exact constant $(2\sqrt{2\pi})^{-1}=0.1994\ldots$ stated in~\eqref{LowAEX-low-bound} turns to be less accurate in Rozovskii-type inequality, than the lower bound of the asymptotically best constant in~\eqref{ABERozLowBound}, since the minorant in~\eqref{ABERozLowBound} is monotone with respect to $\gamma$ for $\eps\le1$ and with respect to $\gamma/\eps$ for $\eps>1$ varying within the range from $(\sqrt{2\pi})^{-1}$as $\gamma/(\eps\vee1)\to\frac23$ to
$$
\frac{1 + \sqrt{5}}{3\sqrt{2\pi} \big(3 - \sqrt{5}\big)} =0.5633\ldots \quad\text{as}\quad \frac{\gamma}{\eps\vee1}\to0
$$
and thus staying always greater than $(2\sqrt{2\pi})^{-1}$. The minorant to the asymptotically best constant in~\eqref{ABEEsseenLowBound} monotonically decreases with respect to $\gamma$ and does not depend on $\eps$. Therefore, there exists a unique value $\gamma_0>0$ such that $\underline{\abe}(\gamma_0)=(2\sqrt{2\pi})^{-1}$ (and even $\gamma_0>1$ due to that for $\gamma=1$ we have  $\underline{\abe}(1)=(\sqrt{10}+3)/(6\sqrt{2\pi})=0.4097\ldots>(2\sqrt{2\pi})^{-1}$), so that for all $\gamma<\gamma_0$ we have $\underline{\abe}(\gamma)>(2\sqrt{2\pi})^{-1}$. 
It is easy to make sure that $\gamma_0= 4.7010\ldots$\ . Therefore, as a lower bound for the constant $\aex(g_0,\eps,\gamma)$ in Esseen-type inequality~\eqref{AW_ineq} it is reasonable to choose the lower bound in~\eqref{LowAEX-low-bound} for $\gamma\le\gamma_0$ and the lower bound in~\eqref{ABEEsseenLowBound} for $\gamma>\gamma_0$. Let us formulate this as a corollary.

\begin{cor}
For the asymptotically exact constant in inequality~\eqref{AW_ineq} the lower bound
$$
\inf_{\eps>0}\aex(g_0,\eps,\gamma)\ge \underline{\abe}(\gamma\wedge\gamma_0), 
$$ 
holds, where $\underline{\abe}(\gamma)$ is defined in~\eqref{ABEEsseenLowBound}, and $\gamma_0=4.7010\ldots$ is the unique root of the equation $\underline{\abe}(\gamma)=(2\sqrt{2\pi})^{-1}$ for $\gamma>0$. In particular, with the account of table~\ref{Tab:Ess+RozAEX(eps,gamma)<=} the following two-sided bounds hold:
$$
0.4097\ldots=\tfrac{\sqrt{10}+3}{6\sqrt{2\pi}}\le\aex(g_0,1,1)\le1.80.
$$

For the asymptotically exact constant in~\eqref{Rozovsky_ineq} we have
$$
\aex(g_0,\eps,\gamma)\ge
\begin{cases}
\displaystyle
\frac{1 + \sqrt{5}}{3\sqrt{2\pi} \big(2\gamma(\eps^{-1}\wedge1)(\sqrt{5} - 2) + 3 - \sqrt{5}\big)}, &\displaystyle
\frac\gamma{\eps\vee1} <\frac{2}{3}, 
\\[3mm]
\displaystyle\frac{1}{\sqrt{2\pi}},
&\displaystyle
\frac\gamma{\eps\vee1} \ge \frac{2}{3}.
\end{cases}
$$
In particular, with the account of table~\ref{Tab:Ess+RozAEX(eps,gamma)<=}  the following two-sided bounds hold:
$$
0.3989\ldots=\tfrac{1}{\sqrt{2\pi}}\le\aex(g_0,1,1)\le1.80.
$$
\end{cor}

\begin{proof}[Proof of theorem~\ref{theorem_lae_lb}]
The relations between the constants follow from their definitions (see also~\eqref{asypt_const_in}). Therefore, it remains to prove the lower bounds for $\lowaex(g_0, \eps, \gamma)$, $\upaexcond(g_0, \eps, \gamma)$, and $\upaexcond(g_1, \eps, \gamma)$.

Following~\cite[\S\,2.3.2]{Shevtsova2016}, consider i.i.d. r.v.'s $X_1,\ldots, X_n$ with a symmetric tree-point distribution
$$
 \Prob(\abs{X_1} = 1) = p = 1 - \Prob(X_1 = 0) \in (0,1),
$$
whose d.f. will be denoted by $F_p(x)=\Prob(X_1<x),$ $x\in\R$. Then
$$
\E X_1 = 0, \quad \E X_1^2 = p,\quad 
\sigma_1^2(z)=\E X_1^2\I(|X_1|\ge z)
= \begin{cases}
p,& z \le 1, 
\\
0,& z > 1,
\end{cases}
\quad \mu_1(\,\cdot\,)\equiv0,
$$
$$
B_n^2=np,\quad g_0(z)=\min\{z,\sqrt{np}\},\quad g_1(z)=\max\{z,\sqrt{np}\},
$$
and the fractions $\es$, $\roz$ coincide, do not depend on $\gamma$ and take the form
$$
L_n(g,\eps)\coloneqq \es(g,\eps,\,\cdot\,) =\roz(g,\eps,\,\cdot\,) 
=\sup_{0<z < \eps B_n} 
\frac{g(z)}{B_n^2 g(B_n)}\sum_{k=1}^n \sigma_k^2(z)
%= \sup_{0 < z < \eps}\frac{g(zB_n)L_n(z)}{g(B_n)}.
=\sup_{0<z < \eps\sqrt{np}} 
\frac{g(z)\sigma_1^2(z)}{pg(\sqrt{np})}.
$$
Due to the monotonicity of $g\in\G$, we have
$$
\sup_{0 < z < \eps \sqrt{np}} g(z)\sigma_1^2(z)
=pg(\eps \sqrt{np}\wedge1)
%= \begin{cases}
%g(1) \cdot p, \quad \text{если} \,\,\, \eps \sqrt{np} > 1,
%\\
%g(\eps \sqrt{np}) \cdot p, \quad \text{если} \,\,\, \eps \sqrt{np} \le 1
%\end{cases} =  
%$$
%$$
= p \, \min\{g(1), g(\eps \sqrt{np})\},
$$
and hence,
$$
L_n(g,\eps) = \frac{\min\left\{g(1),g(\eps\sqrt{np})\right\}}{g(\sqrt{np})},\quad \eps>0,\ n\in\N,\ p\in(0,1),
$$
in particular,
$$
L_n(g_0,\eps)
= \min\left\{ \frac{1}{\sqrt{np}}, \eps,1 \right\}, \quad L_n(g_1,\eps) = \min\left\{ \frac{1}{\sqrt{np}}\vee1,\,\eps\vee1\right\},\quad \eps>0,
$$
and $L_n(g_1,\eps)=1$ for $\eps\le1.$ Moreover, for all $n > \ell^{-2} \vee \eps^{-2}\vee1$ we have
$$
\left\{p\in(0,1): L_n(g_0,\eps) = \ell\right\} = 
\Big\{p\in(0,1): \min\Big\{\frac{1}{\sqrt{np}}, \eps,1\Big\} = \ell\Big\} =$$
$$=\begin{cases}
\left\{\ell^{-2}/n\eqqcolon p(\ell)\right\}, & \ell < \eps\wedge1,
\\
\big(0, (\eps^{-2}\vee1)/n\big), & \ell = \eps\wedge1,
\\
\emptyset, & \ell > \eps\wedge1,
\end{cases}
\quad \text{for all } \eps>0,
$$
$$
\left\{p\in(0,1): L_n(g_1,\eps) = \ell\right\}
=\begin{cases}
\left\{p: \ell = 1\right\}
=\begin{cases}
(0, 1), & \ell = 1,
\\
\emptyset, & \text{otherwise},
\end{cases}
& \eps\le1,
\\
\Big\{p:\min\big\{\eps,\frac{1}{\sqrt{np}}\vee1\Big\}=\ell\big\}
=\begin{cases}
p(\ell), & \ell\in[1,\eps],
\\
\emptyset, & \text{otherwise},
\end{cases}
& \eps>1.
\end{cases}
$$
Due to that the fractions do not depend on $\gamma$, we obtain the lower bounds
$$
\inf_{\gamma>0} \lowaex(g_0, \eps, \gamma) \ge  \limsup_{\ell\to0}\limsup_{n\to\infty}\sup_{p\in(0,1)\colon L_n(g_0,\eps)=\ell}\Delta_n(F_p)/\ell
=\limsup_{\ell\to0} \limsup_{n\to\infty } \Delta_n(F_{p(\ell)}) / \ell,
$$
$$
\inf_{\gamma>0} \upaexcond(g, \eps, \gamma) \ge  \sup_{\ell>0}\limsup_{n\to\infty}\sup_{p\in(0,1)\colon L_n(g,\eps)=\ell}\Delta_n(F_p)/\ell,\quad \eps>0,
$$
in particular, for all $\gamma>0$ we have
$$
\upaexcond(g_0, \eps, \gamma) 
\ge \max\Big\{\sup_{0 < \ell < \eps\wedge1} \limsup_{n \rightarrow \infty} \Delta_n(F_{p(\ell)}) / \ell, \quad \limsup_{n \rightarrow \infty} \sup_{0 < p \le (\eps^{-2}\vee1)/n}\Delta_n(F_{p}) / (\eps\wedge1)\Big\},
$$
$$
\upaexcond(g_1, \eps, \gamma) \ge 
\begin{cases}
\displaystyle
\limsup_{n \rightarrow \infty} \sup_{0 < p < 1}\Delta_n(F_{p}),
&\eps\le1,
\\
\displaystyle
\sup_{1<\ell<\eps}\limsup_{n\to\infty}\Delta_n(F_{p(\ell)})/\ell,
&\eps>1.
\end{cases}
$$

Let us find the lower bound for the uniform distance $\Delta_n(F_p)$. Due to the symmetry, we have $2\Prob(S_n < 0) + \Prob(S_n = 0) = 1$, that is,  $\Prob(S_n < 0)=(1-\Prob(S_n = 0))/2$, whence for even $n$ we obtain
$$
\Delta_n(F_p) \ge \Phi(0) - \Prob(S_n < 0) = \frac{\Prob(S_n = 0)}{2} = \frac{(1-p)^n}{2} \sum_{k=0}^{n/2} \frac{n!}{(n-2k)!(k!)^2}\left(\frac{p/2}{1-p}\right)^{2k}.
$$
Note that $\limsup\limits_{p\to0}\Delta_n(F_p) \ge\tfrac12,$ while with $p=\alpha/n$, $\alpha\in(0,n)$ we have
$$
\Delta_n(F_{\alpha/n}) \ge \frac{1}{2}\left(1-\frac{\alpha}{n}\right)^n \sum_{k=0}^{n/2} \frac{n!}{(n-2k)!(k!)^2} \left(\frac{1/2}{n/\alpha-1}\right)^{2k}
= \frac{e^{-\alpha} + \delta_n(\alpha)}{2} \sum_{k=0}^{n/2}u_{k,n}(\alpha),
$$
where
$$
u_{k,n}(\alpha) = \frac{n!}{(n-2k)!(k!)^2} \left(\frac{1/2}{n/\alpha - 1}\right)^{2k}, \quad \lim_{n \rightarrow \infty}\delta_n(\alpha) = 0,\quad \alpha>0.
$$
In~\cite[p.\,268--269]{Shevtsova2016} it was shown that for every $\alpha>0$
$$
\limsup_{n \rightarrow \infty} \sum_{k=0}^{n/2}u_{k,n}(\alpha) \ge \sum_{k=0}^{\infty} \frac{(\alpha/2)^{2k}}{(k!)^2} = I_0(\alpha).
$$
Therefore,
\begin{equation}\label{lim_nDelta_n(F_a/n)}
\limsup_{n \rightarrow \infty} \Delta_n(F_{\alpha/n})\ge  \tfrac{1}{2}e^{-\alpha}I_0(\alpha),\quad \ell>0.
\end{equation}
%В частности, при $p=p(\ell) = \ell^{-2}/n$ получаем
%$$
%\limsup_{n\to\infty} \Delta_n(F_{p(\ell)}) \ge \tfrac12e^{-\ell^{-2}}I_0(\ell^{-2}),
%$$
Let us bound from below expressions like $\sup_p\Delta_n(F_{p})$ as
\begin{equation}\label{sup_pDelta_n(F_p)}
\sup_{p\in(0,\,\cdot\,)}\Delta_n(F_{p})\ge \limsup_{p\to0}\Delta_n(F_{p})\ge\tfrac12.
\end{equation}
%$$
%\limsup_{n\to\infty}\sup_{0 < p<1}\Delta_n(F_{p})\ge \limsup_{n\to\infty}\sup_{0<p\le\alpha/n}\Delta_n(F_{p})\ge \limsup_{n\to\infty}\lim_{p\to0}\Delta_n(F_p)=\tfrac{1}{2},
%\quad \alpha>0.
%$$
From~\eqref{lim_nDelta_n(F_a/n)} with $\alpha=\ell^{-2}$ we obtain
%Отсюда находим нижние оценки
$$
\lowaex(g_0, \eps, \gamma) \ge \limsup_{\ell\to0}\limsup_{n\to\infty } \Delta_n(F_{p(\ell)}) / \ell
\ge
\tfrac12\lim_{\alpha\to\infty} \sqrt\alpha e^{-\alpha}I_0(\alpha) = \frac{1}{2\sqrt{2\pi}},\quad \eps,\gamma>0.
$$
Inequalities~\eqref{lim_nDelta_n(F_a/n)} and \eqref{sup_pDelta_n(F_p)} imply that
$$
\upaexcond(g_0, \eps, \gamma) 
\ge \max\Big\{\sup_{0 < \ell < \eps\wedge1} \limsup_{n \rightarrow \infty} \Delta_n(F_{p(\ell)}) / \ell, \ \limsup_{n \rightarrow \infty} \sup_{0 < p \le (\eps^{-2}\vee1)/n}\Delta_n(F_{p}) / (\eps\wedge1)\Big\}\ge
$$
$$
\ge \max\Big\{\sup_{\alpha>\eps^{-2}\vee1} \tfrac12\sqrt\alpha e^{-\alpha}I_0(\alpha),\
\tfrac12(\eps^{-1}\vee1)\Big\},
$$
$$
\upaexcond(g_1, \eps, \gamma) \ge 
\begin{cases}
\displaystyle
\limsup_{n \rightarrow \infty} \sup_{0 < p < 1}\Delta_n(F_{p})\ge\frac12,
&\eps\le1,
\\
\displaystyle
\sup_{1<\ell<\eps}\limsup_{n\to\infty}\Delta_n(F_{p(\ell)})/\ell\ge \sup_{\eps^{-2}<\alpha<1}\tfrac{1}{2}\sqrt\alpha e^{-\alpha}I_0(\alpha),
&\eps>1.
\end{cases}
$$
The plot of the function $f(\alpha)=\sqrt{\alpha}e^{-\alpha}I_0(\alpha)$ looks monotonically increasing for $\alpha\le0.78$ and monotonically decreasing for $\alpha\ge0.79\eqqcolon\alpha_*$ with $f(\alpha_*)>0.4688,$ therefore it is reasonable to estimate upper bounds $\sup_\alpha f(\alpha)$ from below as
$$
\sup_{\alpha>\eps^{-2}\vee1}f(\alpha)\ge f(1)=e^{-1}I_0(1)=0.4657\ldots,
$$
$$
\sup_{\eps^{-2}<\alpha<1}f(\alpha)\ge
\begin{cases}
 f(\alpha_*)>0.4688,& \eps^{-2}<\alpha_*,\quad\Leftrightarrow\quad \eps>1/\sqrt{\alpha_*}=1.1250\ldots,
\\
f(\eps^{-2})=\exp\{-\eps^{-2}\}I_0(\eps^{-2}),& \eps^{-2}\ge\alpha_*,\quad\Leftrightarrow\quad \eps\le1/\sqrt{\alpha_*}.
\end{cases}
$$
Hence, we finally obtain
$$
\upaexcond(g_0, \eps, \gamma) 
\ge \tfrac12\max\{0.4657,\eps^{-1},1\}=\tfrac12(\eps^{-1}\vee1),
$$
$$
\upaexcond(g_1, \eps, \gamma) \ge 
\begin{cases}
\tfrac12\exp\{-\eps^{-2}\}I_0(\eps^{-2}),
&1<\eps\le1.1251,
\\
\tfrac12\cdot0.4688=0.2344,&\eps\ge1.1251.\qquad\qedhere
\end{cases}
$$
\end{proof}

\section{Conclusion}

In the present work we introduced a detailed classification of the exact and asymptotically exact constants in natural convergence rate estimates in the Lindeberg's theorem like Esseen's and Rozovskii's inequalities. We found the lower bounds of the exact (universal) constants and the most optimistic absolute constants. The obtained lower bounds turned to be rather close to the upper ones which justifies the high accuracy of the latest.  Also we constructed the lower bounds for the asymptotically best, the lower asymptotically exact and the conditional upper asymptotically exact  constants. With the account of the previously known upper estimates of the asymptotically exact constants, this allowed to obtain two-sided bounds for all the introduced asymptotic constants.

\bibliography{../../../bib/my_pub,../../../bib/biblio}
\bibliographystyle{plain}

\end{document}